\documentclass[11pt]{article}
\usepackage{amsthm,amsfonts, amsbsy, amssymb,amsmath,graphicx}
\usepackage{graphics}
\usepackage[english]{babel}
\usepackage{graphicx}
\usepackage{lineno}
\usepackage{enumerate}
\usepackage{tikz}
\usepackage{tkz-graph}
\usepackage{tkz-berge}
\usepackage{hyperref}
\usepackage{color}
\usepackage{tikz}
\usetikzlibrary{arrows, automata}
\hypersetup{colorlinks=true}

\hypersetup{colorlinks=true, linkcolor=blue, citecolor=blue,urlcolor=blue}
\title{$e$-injective coloring: injective and $2$-distance colorings conjectures}
\author{{\small  Shahrzad Sadat Mirdamad$^{a}$, Doost Ali Mojdeh$^{b}$\thanks{Corresponding author}}\\ {\small $^{a,b}$Department of
		Mathematics, Faculty of Mathematical Sciences}\\{\small University of Mazandaran, Babolsar, Iran}\\
	 {\small $^a$s.mirdamad03@umail.umz.ac.ir}\\ {\small $^b$damojdeh@umz.ac.ir}}
\date{}
\newtheorem{theorem}{Theorem}[section]
\newtheorem{corollary}[theorem]{Corollary}
\newtheorem{proposition}[theorem]{Proposition}
\newtheorem{lemma}[theorem]{Lemma}

\newtheorem{observation}[theorem]{Observation}
\newtheorem{problem}[theorem]{Problem}

\theoremstyle{definition}
\newtheorem{definition}[theorem]{Definition}
\theoremstyle{remark}

\begin{document}
	
	\maketitle
\begin{abstract}
An injective coloring of a given graph $G=(V,E)$ is a vertex coloring of $G$ such that any two vertices with common neighbor receive distinct colors.
An $e$-injective coloring of a graph $G$ is a vertex coloring of $G$ such that any two vertices with common edge neighbor receive distinct colors; in the other words, if $u$ and $v$ are the end of a path $P_4$ in a graph $G$, then they are assigned with different labels.
With this new definition, we want to take a review at injective coloring of a graph from the new point of view.
For this purpose, we will have a comparison between $e$-injective coloring with usual coloring, injective coloring, and $2$-distance coloring.
As well, we review  the conjectures raised so far in the literature of injective coloring and $2$-distance coloring, from the new approach, $e$-injective coloring.  Finally, we precisely investigate  the $e$-injective coloring of trees, join of two graphs, a family of standard graphs, grid graphs, cylinder graphs and tori graphs.
\end{abstract}
\vspace{1mm}
\section{Introduction}
Graph coloring has many applications in various fields of life, such as timetabling, scheduling daily life activities, scheduling computer processes, registering allocations to different institutions and libraries, manufacturing tools, printed circuit testing, routing and wavelength, bag rationalization for a food manufacturer, satellite range scheduling, and frequency assignment. These are some applications out of the many that already exist and many to come.

A proper $k$-coloring (here after, $k$-coloring)  of a graph $G$ is a function $f:V(G)\rightarrow \lbrace 1,2,3,\dots,k \rbrace$
such that for all edge $xy\in E$, $f(x)\neq f(y)$. The chromatic number of $G$, denoted by $\chi(G)$, is the minimum integer $k$ such that $G$ has a $k$-coloring.

In $1969$, Kramer and Kramer \cite{Kramer1} introduced the notion of $2$-{\it distance coloring} of a graph $G$ or the usual proper coloring of $G^2$\ (we can see also their survey on general distance coloring in \cite{Kramer2}) and we see some of its applications in \cite{Halldórsson, La-val}.
A $2$-{\it distance $k$-coloring} of a graph $G$ is a function $f:V\rightarrow \lbrace 1,2,3,\dots,k \rbrace$, such that no pair of vertices at distance at most $2$, receive the same color, in the other words, the colors of the vertices of all $P_3$ paths in the graph are distinct.
The $2$-distance chromatic number of $G$, denoted by $\chi_2(G)=\chi(G^2)$, is the minimum positive integer $k$ such that $G$ has a $P_{3}$-$k$-coloring.
The $2$-distance coloring of $G$, is a proper coloring.

For a graph $G$, the subset $S$ of $V(G)$ is said to  be a dominating set if any vertex $x\in V\setminus S$ is adjacent to a vertex $y$ in $S$. A dominating set of $G$ with minimum cardinality is called the domination number of $G$ and is denoted by $\gamma(G)$. A subset $D$ of $V(G)$ is said to be $2$-distance dominating set if any vertex $d\in V\setminus D$, is in at most $2$-distance of to a vertex in $D$. A $2$-distance dominating set of $G$ with minimum cardinality is called the $2$-distance domination number of $G$ and is denoted by $\gamma_2(G)$.

The injective coloring was first introduced in $2002$ by Hahn et al. \cite{Han-Kratochvil} and it was also further studied in \cite{Bu, Cranston, Luzar, Panda, Song}. In \cite{Bresar}, Brešar et al.  studied the concept of injective coloring from two novel perspectives (related to open packing and the two-step graph operation) and with this view, there were proved several general bounds on the injective chromatic number. An injective $k$-coloring of a graph $G$ is a function $f:V\rightarrow \lbrace 1,2,3,\dots,k \rbrace$ such that no vertex $v$ is adjacent to two vertices $u$ and $w$ with $f(u)=f(w)$, in the other words, for any path $P_3=xyz$, we have  $f(x)\neq f(z)$.
The injective chromatic number of $G$, denoted by $\chi_i(G)$, is the minimum positive integer $k$ such that $G$ has an injective $k$-coloring. The injective chromatic number of the hypercube has important applications in the theory of error-correcting codes.
As it is well known, the injective coloring of $G$, is not necessarily proper coloring. Injective coloring of a graph $G$ is related to the usual coloring of the square $G^2$. The inequality $\chi_i(G)\le  \chi_2(G)$ obviously
holds.

There are several results related to injective coloring that review the usual coloring results in graph theory from a new point of view, in particular on the  injective chromatic number of planar graphs.
  As well, many conjectures on planar graphs have
been posed and studied by authors so far. In this regard, we  bring up some of them as follows.\\

From the relation between the injective coloring of a graph $G$ and the usual coloring of the square $G^2$, Wegner \cite{Wegner} posed the following conjecture.\\

\noindent\textbf{Conjecture 1.} \cite{Wegner} Let $G$ be a planar graph with maximum degree $\Delta$. The chromatic number of $G^2$ is at most

1. $7$, if $\Delta=3$,

2. $\Delta+5$, if $4 \le \Delta \le 7$,

3. $\lfloor\frac{3}{2} \Delta\rfloor + 1$, otherwise.\\

Lu\v{z}ar, R. \v{S}krekovski in \cite{Luzar2} showed  that,

\noindent\textbf{Theorem A.} (\cite{Luzar2} Theorem 2.1) There exist planar graphs $G$ of maximum degree $\Delta \ge 3$ satisfying the
following\\

1. $\chi_i(G)=5$, if $\Delta=3$,

2. $\chi_i(G)=\Delta+5$, if $4 \le \Delta \le 7$,

3. $\chi_i(G)=\lfloor\frac{3}{2} \Delta\rfloor + 1$, if $\Delta \ge 8$.\\

Adapted from Theorem A, they proposed the following Wegner type conjecture for the injective chromatic number of planar graphs.\\

\noindent\textbf{Conjecture 2.} \cite{Luzar2} Let $G$ be a planar graph with maximum degree $\Delta$. The injective chromatic number $\chi_i(G)$ of $G$ is at most

(i) $5$, if $\Delta=3$,

(ii) $\Delta+5$, if $4 \le \Delta \le 7$,

(iii) $\lfloor\frac{3}{2} \Delta\rfloor + 1$, otherwise.\\

By the relation between injective chromatic number and  $2$-distance chromatic number of a graph; showing the truth of  Conjecture 1 (parts (2), (3)), will deduce  the truth of Conjecture 2 (parts (ii), (iii)).

\indent Now we introduce a new concept of vertex coloring (near to injective coloring) as an $e$-injective coloring of a graph. The motivation of the alleging is to study, how it behaves against of the injective graph coloring, usual graph coloring, $2$-distance graph coloring,  packing set and $2$-distance dominating set of graphs.
As well, in particular we investigate the posed conjectures from the point of view of $e$-injective colorings.
Also since the notion of $e$-injective  coloring is near to injective coloring, one can predict, it has applications in various fields of life in real world and would also be useful in coding theory as so did injective coloring.

This concept is introduced as follows.

\begin{definition}
Let $G$ be a graph.	A function $f:V(G)\rightarrow \lbrace 1,2,3,\dots,k \rbrace$ is an $e$-injective $k$-coloring function if any two vertices with common edge neighbor receive distinct colors under $f$; in the other words, if $u$ and $v$ are the end of path $P_4=uxyv$ in  $G$, then $f(u)\ne f(v)$.

The  $e$-{\it injective chromatic number} of $G$, denoted by $\chi_{ei}(G)$, is the minimum positive integer $k$ such that $G$ has an $e$-injective $k$-coloring.

The $e$-injective coloring of $G$, is not necessarily proper coloring.
\end{definition}

This concept can be expressed as new discourse.

\begin{definition}
	For a given graph $G$, {\it the three-step graph} $S_{3}(G)=G^{(3)}$ of a graph $G$ is the graph having the same vertex set as $G$ with an edge joining two vertices in $S_{3}(G)$ if and only if  there is a path of length $3$ between them in $G$.\\
Taking into account, the fact that a vertex subset $S$ is independent in $S_{3}(G)$ if and only if there is no path of length $3$ between any two vertices corresponding of $S$ in $G$,  we can readily observe that:
$$\chi_{ei}(G)=\chi(S_{3}(G))$$
\end{definition}

From the point of view of $e$-injective coloring,  the type of  conjectures 1 and 2 can be declared as a problem, which can be argued in next section.\\

\begin{problem}\label{prob-conj}
Let $G$ be a planar graph with maximum degree $\Delta$. Then $\chi_{ei}(G)$ is at most

\emph{1}. $5$, if $\Delta=3$,

\emph{2}. $\Delta+5$, if $4 \le \Delta \le 7$,

\emph{3}. $\lfloor\frac{3}{2} \Delta\rfloor + 1$, otherwise.\\
\end{problem}

In the sequel of this section, we introduce some notations. We use \cite{chartrand, west} as a reference for terminology and notation which are not explicitly defined here. Throughout the paper, we consider $G=(V,E)$ be a finite simple graph with vertex set $V=V(G)$ and edge set $E=E(G)$.
The open neighborhood of a vertex $v\in V$ is the set $N(v)=\lbrace u \lvert uv \in E \rbrace$, and its closed neighborhood is the set $N[v]=N(v) \cup \lbrace v \rbrace$.
The cardinality of $\lvert N(v) \rvert$ is called the degree of $v$, denoted by $\deg(v)$.
The minimum degree of $G$ is denoted by $\delta (G)$ and the maximum degree by $\Delta (G)$.
A vertex $v$ of degree $1$ is called a pendant vertex or a leaf, and its neighbor is called a support vertex.
A vertex of degree $n-1$ is called a full or universal vertex while a vertex of degree $0$ is called an isolated vertex.
For any two vertices $u$ and $v$ of $G$, we denote by $d_G(u,v)$ the distance between $u$ and $v$, that is the length of a shortest path joining $u$ and $v$.
The path, cycle and complete  graph with $n$ vertices are denoted by $P_{n}$, $C_{n}$ and $K_{n}$ respectively.
The complete bipartite graph with $n$ and $m$ vertices in their partite sets is denoted by $K_{n,m}$, while the wheel graph with $n+1$ vertices is denoted by $W_{n}$.
A star graph with $n+1$ vertices, denoted by $S_{n}$, consists of $n$ leaves and one support vertex.
A double star graph is a graph consisting of the union of two star graphs $S_{m}$ and $S_{n}$,  with one edge joining their support vertices; the double star graph with $m+n+2$ vertices is denoted by $S_{m,n}$.

The join of two graphs {$G$} and {$H$}, denoted by $G\vee H$, is the graph obtained from the disjoint union of $G$ and $H$ with vertex set $V(G\vee H)
=V(G)\cup V(H)$ and edge set $E(G\vee H)= E(G)\cup E(H)\cup \{xy\ |\ x\in V(G),\ y\in V(H)\}$.
A fan graph is a simple graph consisting of joining $\overline{K}_{m}$ and $P_{n}$; the fan graph with $m+n$ vertices is denoted by $F_{m,n}$. For two sets of vertices $X$ and $Y$, the set $[X,Y]$ denotes the set of edges $e=uv$ such that $u\in X$ and $v\in Y$.

The square graph $G^{2}$ is a graph with the same vertex set as $G$ and with its edge set given by $E(G^{k})=\{uv| \ dist(u,v)\leq 2\}$. The chromatic number $\chi(G^2)$ of $G^{2}$ (or $2$-distance chromatic number $\chi_2(G)$ of $G$) has been studied extensively in planar graph \cite{Heuvel, Molloy}. One of the important applications of the chromatic number of $G^{2}$ is in the field of Steganography, which is the study of hiding messages within an ordinary and non-secret message to avoid detection \cite{Fridrich}.

A subset $B\subseteq V(G)$ is a packing set in $G$ if for every pair of distinct vertices $u, v\in B$, $N_{G}[u]\bigcap N_{G}[v]=\emptyset$.
The packing number $\rho (G)$ is the maximum cardinality of a packing set in $G$.\\
A subset $B\subseteq V(G)$ is an open packing set in $G$ if for every pair of distinct vertices $u, v\in B$, $N_{G}(u)\bigcap N_{G}(v)=\emptyset$.
The open packing number $\rho^{\circ}(G)$ is the maximum cardinality  among all open packing sets in $G$.

Let $G$ and $H$ be simple graphs. For three standard products of graphs $G$ and $H$, the vertex set of the product is $V(G)\times V(H)$ and their edge set is defined as follows.\\
$\bullet$ In the {\it Cartesian product} $G\square H$, two vertices are adjacent if they are adjacent in one coordinate and equal in the other.\\
$\bullet$ In the {\it direct product} $G\times H$, two vertices are adjacent if they are adjacent in both coordinates.\\
$\bullet$ The edge set of the {\it strong product} $G\boxtimes H$, is the union of $E(G\square H)$ and $E(G\times H)$.\\
A {\it grid graph} is a graph obtain from the Cartesian product of paths $P_m$ and $P_n$.
A {\it cylinder graph} is a graph obtain from the Cartesian product of path $P_m$ and cycle $C_n$. A {\it tori graph} is a graph obtain from the Cartesian product of $C_m$ and $C_n$. 

In the end of this section, we explore the purpose of the paper as follows.
In Section 2,  we study $\chi_{ei}(G)$ versus to the $\chi(G)$, $\chi_{2}(G)$ and $\chi_{i}(G)$, as well, we review  the conjectures raised so far in the literature of injective  and $2$-distance colorings, from the new approach, $e$-injective coloring, and by disproving the Problem \ref{prob-conj}, we show that the Conjectures 1 and 2 may be wrong.
We prove that, for any tree $T$ ($T\neq S_n$),
$\chi_{ei}(T)=\chi(T)$, for disjoint graphs $G, H$, with non-empty edge sets, $\chi_{ei}(G\cup H)=max\{\chi_{ei}(G), \chi_{ei}(H)\}$ and $\chi_{ei}(G\vee H)=|V(G)|+ |V(H)|$.
The $e$-injective chromatic number of $G$  versus of the maximum degree and packing number of $G$ is investigated, and denote 
$ max\{\chi_{ei}(G), \chi_{ei}(H)\} \leq \chi_{ei}(G\square H)\leq  \chi_{2}(G)\chi_{2}(H).$
In Section 3, we obtain the exact value of  $e$-injective chromatic number of some standard graphs and finally, we argue 
on the $e$-injective coloring of grid graphs, cylinder graphs and tori graphs and gain their $e$-injective chromatic number in Section 4. We end the paper with  discussion on research problems.

\section{General results}
We maybe cannot compare $\chi_{ei}(G)$ with $\chi(G)$, $\Delta(G)$, $\chi_i(G)$ or $\chi_2(G)$ in general. For instance,  $\chi_{ei}(K_3)=1$
while $\chi(K_3)=3=\chi_i(K_3)=\chi_2(G)$; or $\chi_{ei}(K_{1,n})=1$ while $\chi(K_{1,n})=2,\ \chi_i(K_{1,n})=n$ and $\chi_2(K_{1,n})=n+1$.
But in the figure \ref{Fig-ex1}, $\chi_{ei}(G)= 12$, $\chi_{i}(G)\le7$, $\chi_{2}(G)\le7$, $\chi(G)= 3$.

\begin{figure}[h]
\centering
\begin{tikzpicture}[scale=.4, transform shape]
		\node [scale=1.3] at (0.3,2) {$5$};
		\node [scale=1.3] at (-1.2,2) {$4$};
		\node [scale=1.3] at (1.4,1.9) {$6$};
		\node [scale=1.3] at (3.8,0.8) {$7$};
		\node [scale=1.3] at (-2.7,2) {$3$};
        \node [scale=1.3] at (0.3,-0.5) {$11$};
        \node [scale=1.3] at (-2.7,-0.5) {$9$};
        \node [scale=1.3] at (1.8,-0.5) {$12$};
        \node [scale=1.3] at (-1.2,-0.5) {$10$};
        \node [scale=1.3] at (-4.2,2) {$2$};
        \node [scale=1.3] at (-6.2,0.8) {$1$};
        \node [scale=1.3] at (-4,-0.4) {$8$};

        \draw[line width=1pt](-4.2,1.5) - - (-4.2,0);
        \draw[line width=1pt](-4.2,1.5) - - (-2.7,1.5);
        \draw[line width=1pt](-4.2,0) - - (-2.7,0);
        \draw[line width=1pt](-2.7,1.5) - - (-4.2,0);
        \draw[line width=1pt](-4.2,1.5) - - (-5.7,0.75);
        \draw[line width=1pt](-4.2,0) - - (-5.7,0.75);
        \draw[line width=1pt](0.3,1.5) - - (0.3,0);
        \draw[line width=1pt](-2.7,1.5) - - (-2.7,0);
        \draw[line width=1pt](0.3,1.5) - - (-1.2,1.5);
        \draw[line width=1pt](0.3,0) - - (-2.7,0);
        \draw[line width=1pt](0.3,0) - - (1.8,0);
        \draw[line width=1pt](0.3,1.5) - - (1.8,1.5);
        \draw[line width=1pt](1.8,1.5) - - (1.8,0);
        \draw[line width=1pt](1.8,1.5) - - (3.3,0.75);
        \draw[line width=1pt](1.8,0) - - (3.3,0.75);
        \draw[line width=1pt](-1.2,1.5) - - (-2.7,0);
        \draw[line width=1pt](-1.2,1.5) - - (-2.7,1.5);
        \draw[line width=1pt](0.3,0) - - (1.8,1.5);
        \draw[line width=1pt](-1.2,0) - - (-1.2,1.5);
        \draw[line width=1pt](-1.2,0) - - (0.3,1.5);
        \draw[line width=1pt](3.3,0.75) to [bend right=45 ](-4.2,1.5);
        \draw[line width=1pt](-5.7,0.75) to [bend right=50 ](1.8,0);

		\node [draw, shape=circle,fill=white] (f) at  (-5.7,0.75) {};
	    \node [draw, shape=circle,fill=white] (a) at  (-4.2,1.5) {};
        \node [draw, shape=circle,fill=white] (b) at  (-2.7,1.5) {};
        \node [draw, shape=circle,fill=white] (c) at  (-1.2,1.5) {};
	    \node [draw, shape=circle,fill=white] (h) at  (-4.2,0) {};
        \node [draw, shape=circle,fill=white] (i) at  (-2.7,0) {};
        \node [draw, shape=circle,fill=white] (j) at  (-1.2,0) {};
	    \node [draw, shape=circle,fill=white] (d) at  (0.3,1.5) {};
        \node [draw, shape=circle,fill=white] (k) at  (0.3,0) {};
   		\node [draw, shape=circle,fill=white] (e) at  (1.8,1.5) {};
		\node [draw, shape=circle,fill=white] (l) at  (1.8,0) {};
		\node [draw, shape=circle,fill=white] (g) at  (3.3,0.75) {};

	\end{tikzpicture}
\caption{The graph $G$ with $\chi_{ei}(G)\ge max\{\chi_{i}(G), \chi_{2}(G), \chi(G)\}$}\label{Fig-ex1}
\end{figure}
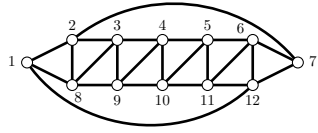

Also  let $H=K_m\circledcirc K_n$ be a graph obtain from two
complete graphs $K_m$ and $K_n$ ($m\ge n\ge 4$) with joining one vertex of $K_m$ to one vertex of $K_n$.
Then $\chi(H)=m=\chi_{i}(H)$, $\chi_{2}(H)=m+1$ and $\chi_{ei}(H)=m+n-1$ (see Observation \ref{obs-kn}).
On the other hand, for Complete graph $K_n$ ($n\ge 4$), odd cycle $C_{3k}$ for odd $k$, we have $\chi_{ei}(K_n)=n=\chi(K_n)=\chi_{i}(K_n)=\chi_{2}(K_n)$,
and $\chi_{ei}(C_{3k})=3=\chi(C_{3k})=\chi_{i}(C_{3k})=\chi_{2}(C_{3k})$.

In the same way, we have
$\Delta(K_3)=2$, $\Delta(H)=m$, and for graph $G$ in Figure \ref{Fig-ex1}, $\Delta(G)=4$,
while $\chi_{ei}(K_3)=1$, $\chi_{ei}(H)=m+n-1$, and
$\chi_{ei}(G)=12$. On the other hand, for  graph $K_n\circledcirc K_1$ ($n\ge 4$) and even cycle $C_{2k}$, we have $\chi_{ei}(K_n\circledcirc K_1)=n
=\Delta(K_n\circledcirc K_1)$, and $\chi_{ei}(C_{2k})= 2=\Delta(C_{2k})$ (see  Proposition \ref{prop-cn}).\\

However we have the following.
\begin{observation}\label{obs-e-uso}
Let $G$ be a graph in which any two adjacent vertices be the end vertices of  a path $P_4$ in $G$. Then $\chi(G) \le \chi_{ei}(G)$.\\
Conversely, if any two end vertices of each path $P_4$ in $G$ are adjacent, then $\chi_{ei}(G) \le \chi(G)$.
\end{observation}

\begin{proof}
Since any two adjacent vertices of given graph $G$ are the end vertices of a path $P_4$ in $G$, these two vertices must be colored with distinct colors by any $e$-injective coloring. Therefore, any $e$-injective coloring for this graph can be a usual coloring. Then $\chi(G) \le \chi_{ei}(G)$.

Conversely, by the construction of graph $G$, usual coloring of $G$ deduces that any two end vertices of each path $P_4$ has different colors. Thus, the usual coloring of given $G$ is an $e$-injective coloring. Therefore, $\chi_{ei}(G) \le \chi(G)$.
\end{proof}

\begin{observation}\label{obs-e-inj}
Let $G$ be a graph and $v$ in $V(G)$ be any vertex.  If  every two vertices in $N(v)$  are the end vertices of a path $P_4$, then $\chi_{i}(G) \le \chi_{ei}(G)$.\\
Conversely, if end vertices of each path $P_4$ in a graph $G$ have a common neighbor, then  $\chi_{ei}(G) \le \chi_{i}(G)$.
\end{observation}

\begin{proof}
Since any two  vertices of graph $G$ are the end vertices of a path $P_4$, these vertices have distinct colors by $e$-injective coloring. Therefore, an $e$-injective coloring for this graph is an injective coloring. Then $\chi_{i}(G) \le \chi_{ei}(G)$.

Conversely, let any two end vertices of  path $P_4$ have a common neighbor. Then by injective coloring of $G$, two end vertices of any path $P_4$ take
different colors. Thus, this will be an $e$-injective coloring of $G$ and $\chi_{ei}(G) \le \chi_{i}(G)$.
\end{proof}

Also, we may have.
\begin{observation}\label{obs-e-dis}
Let $G$ be a graph with the property that, for any three vertices on a path $P_3$ in $G$, any both of them are the end vertices of
a path $P_4$ in $G$, then  $\chi_{2}(G) \le \chi_{ei}(G)$.\\
Conversely, if $G$ is a graph and any two end vertices of each path $P_4$ in $G$ are adjacent or have a common neighbor, then $\chi_{ei}(G) \le \chi_{2}(G)$.
\end{observation}

\begin{proof}
Let $v, u, w$ be three vertices of $P_3$ in $G$. Since both of them are the end vertices of a path $P_4$ in $G$, then $e$-injective coloring of the given graph $G$ assign three distinct colors to $v, u, w$. This implies that, this coloring is a
$2$-distance coloring of $G$. Thus, $\chi_{2}(G) \le \chi_{ei}(G)$.

Conversely, let any two end vertices of each path $P_4$ are adjacent or have a common neighbor. Then, from a $2$-distance coloring of $G$, we deduce that, any two end
vertices of the path $P_4$ are  vertices of a $P_2$ or a $P_3$ in  graph $G$ and so their colors are distinct. Therefore this coloring is an $e$-injective coloring of the given graph $G$ and $\chi_{ei}(G) \le \chi_{2}(G)$.
\end{proof}

Now we discuss on the Problem \ref{prob-conj}. Below figures denote that the Problem \ref{prob-conj} is not necessarily true. On the other hand the type of conjectures 1, 2
are not true for $e$-injective coloring. But if we use the Observations \ref{obs-e-inj}, \ref{obs-e-dis}, then maybe characterize  graphs $G$ in which, satisfy on the Conjectures 1, 2
and also  characterize  graphs $G$ in which, the Conjectures 1, 2  are not true for them.
\vspace{2mm}

\noindent \textbf{Disprove of Problem} 1.3\\
Let $G$ be a planar graph with maximum degree $\Delta$. Then $\chi_{ei}(G)$ is at most

1. $5$, if $\Delta=3$,

2. $\Delta+5$, if $4 \le \Delta \le 7$,

3. $\lfloor\frac{3}{2} \Delta\rfloor + 1$, otherwise.\\

As we observe the Figure \ref{Fig-problem-n}, ($3\le \Delta \le 8$) we have.

The graph $M_3$ denotes a planar graph in which $\Delta(M_3)=3$ and $\chi_{ei}(M_3)=6> 5$.

The graph $M_4$ denotes a planar graph in which $\Delta(M_4)=4$ and $\chi_{ei}(M_4)=12> \Delta(M_4)+5$.

The graph $M_5$ denotes a planar graph in which $\Delta(M_5)=5$ and $\chi_{ei}(M_5)=13 > \Delta(M_5)+5$.

The graph $M_6$ denotes a planar graph in which $\Delta(M_6)=6$ and $\chi_{ei}(M_6)=16> \Delta(M_6)+5$.

The graph $M_7$ denotes a planar graph in which $\Delta(M_7)=7$ and $\chi_{ei}(M_7)=16> \Delta(M_7)+5$.

For $\Delta(G)= 8$, consider the graph $M_8$  of order $16$,  which is seen $\Delta(G)=8$ and $\chi_{ei}(M_n)=16> \lfloor \frac{24}{2} \rfloor+1=13=
\lfloor \frac{3\Delta}{2} \rfloor+1$.
\vspace{3mm}

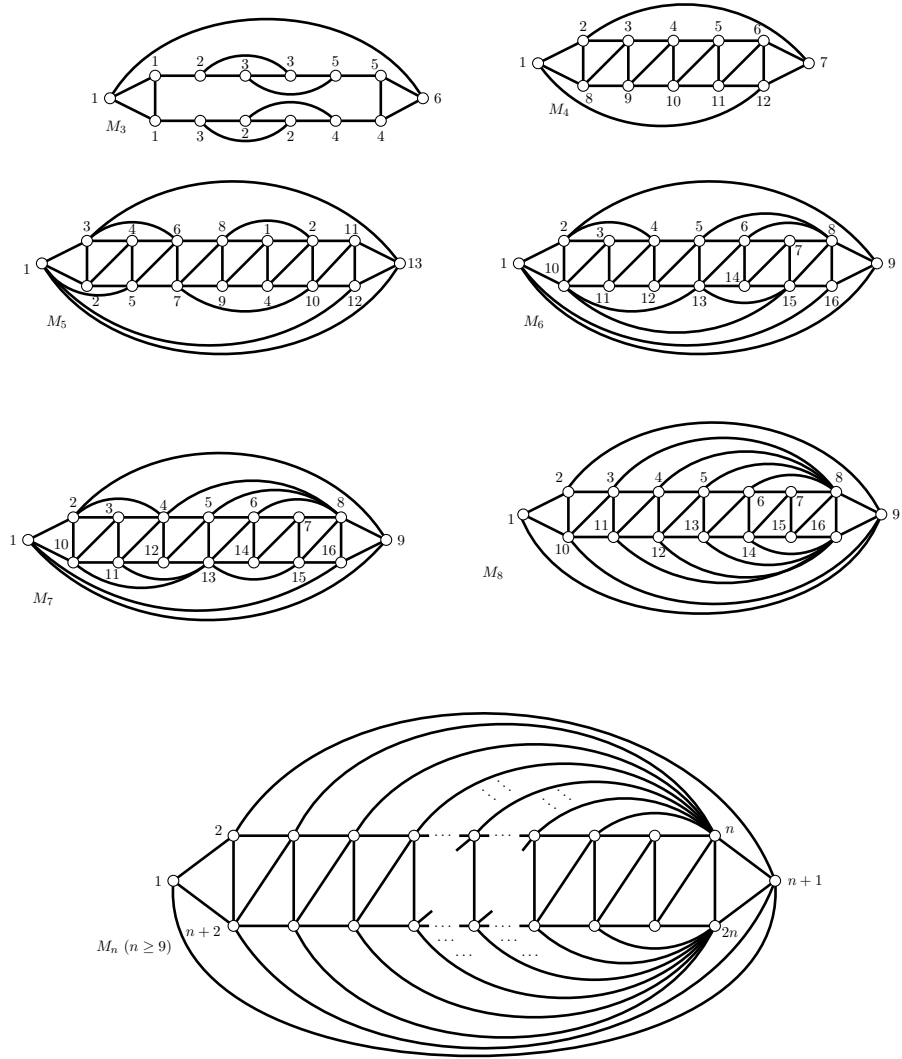
\begin{figure}[h]
\centering
	\begin{tikzpicture}[scale=.4, transform shape]
		\node [scale=1.3] at (-0.7,2) {$5$};
		\node [scale=1.3] at (-2.2,2) {$3$};
		\node [scale=1.3] at (0.6,1.9) {$5$};
		\node [scale=1.3] at (2.7,0.75) {$6$};
		\node [scale=1.3] at (-3.7,1.9) {$3$};
        \node [scale=1.3] at (-0.7,-0.5) {$4$};
        \node [scale=1.3] at (-3.7,-0.4) {$2$};
        \node [scale=1.3] at (0.8,-0.5) {$4$};
        \node [scale=1.3] at (-2.2,-0.5) {$2$};
        \node [scale=1.3] at (-5.2,2) {$2$};
        \node [scale=1.3] at (-8.7,0.75) {$1$};
        \node [scale=1.3] at (-5.2,-0.5) {$3$};
        \node [scale=1.3] at (-6.7,-0.5) {$1$};
        \node [scale=1.3] at (-6.7,2) {$1$};
\node [scale=1.3] at (-8,-.2) {$M_3$};
        \draw[line width=1pt](-8.2,0.75) - - (-6.7,0);
        \draw[line width=1pt](-8.2,0.75) - - (-6.7,1.5);
        \draw[line width=1pt](-6.7,1.5) - - (-5.2,1.5);
        \draw[line width=1pt](-6.7,0) - - (-5.2,0);
        \draw[line width=1pt](-6.7,1.5) - - (-6.7,0);
        \draw[line width=1pt](-5.2,1.5) - - (-3.7,1.5);
        \draw[line width=1pt](-5.2,0) - - (-3.7,0);
        \draw[line width=1pt](-3.7,1.5) - - (-2.2,1.5);
        \draw[line width=1pt](-3.7,0) - - (-2.2,0);
        \draw[line width=1pt](-2.2,1.5) - - (-0.7,1.5);
        \draw[line width=1pt](-2.2,0) - - (-0.7,0);
        \draw[line width=1pt](-0.7,1.5) - - (0.8,1.5);
        \draw[line width=1pt](-0.7,0) - - (0.8,0);
        \draw[line width=1pt](0.8,0) - - (0.8,1.5);
        \draw[line width=1pt](0.8,0) - - (2.2,0.75);
        \draw[line width=1pt](0.8,1.5) - - (2.2,0.75);

        \draw[line width=1pt](-8.2,.75) to [bend left=60 ](2.2,.75);
        \draw[line width=1pt](-0.7,1.5) to [bend left=45 ](-3.7,1.5);
        \draw[line width=1pt](-2.2,1.5) to [bend right=50 ](-5.2,1.5);
        \draw[line width=1pt](-0.7,0) to [bend right=45 ](-3.7,0);
        \draw[line width=1pt](-2.2,0) to [bend left=50 ](-5.2,0);
		\node [draw, shape=circle,fill=white] (g) at  (-8.2,0.75) {};
	    \node [draw, shape=circle,fill=white] (a) at  (-6.7,1.5) {};
	    \node [draw, shape=circle,fill=white] (b) at  (-5.2,1.5) {};
	    \node [draw, shape=circle,fill=white] (c) at  (-3.7,1.5) {};
	    \node [draw, shape=circle,fill=white] (d) at  (-2.2,1.5) {};
	    \node [draw, shape=circle,fill=white] (e) at  (-0.7,1.5) {};
	    \node [draw, shape=circle,fill=white] (f) at  (0.8,1.5) {};
	    \node [draw, shape=circle,fill=white] (i) at  (-6.7,0) {};
	    \node [draw, shape=circle,fill=white] (j) at  (-5.2,0) {};
	    \node [draw, shape=circle,fill=white] (k) at  (-3.7,0) {};
	    \node [draw, shape=circle,fill=white] (l) at  (-2.2,0) {};
	    \node [draw, shape=circle,fill=white] (m) at  (-0.7,0) {};
	    \node [draw, shape=circle,fill=white] (n) at  (0.8,0) {};
	    \node [draw, shape=circle,fill=white] (h) at  (2.2,0.75) {};

	\end{tikzpicture}
\qquad
    \centering
\begin{tikzpicture}[scale=.4, transform shape]
		\node [scale=1.3] at (0.3,2) {$5$};
		\node [scale=1.3] at (-1.2,2) {$4$};
		\node [scale=1.3] at (1.6,1.9) {$6$};
		\node [scale=1.3] at (3.8,0.8) {$7$};
		\node [scale=1.3] at (-2.7,2) {$3$};
        \node [scale=1.3] at (0.3,-0.5) {$11$};
        \node [scale=1.3] at (-2.7,-0.4) {$9$};
        \node [scale=1.3] at (1.8,-0.5) {$12$};
        \node [scale=1.3] at (-1.2,-0.5) {$10$};
        \node [scale=1.3] at (-4.2,2) {$2$};
        \node [scale=1.3] at (-6.2,0.8) {$1$};
        \node [scale=1.3] at (-4,-0.4) {$8$};
\node [scale=1.3] at (-5,-.8) {$M_4$};
        \draw[line width=1pt](-4.2,1.5) - - (-4.2,0);
        \draw[line width=1pt](-4.2,1.5) - - (-2.7,1.5);
        \draw[line width=1pt](-4.2,0) - - (-2.7,0);
        \draw[line width=1pt](-2.7,1.5) - - (-4.2,0);
        \draw[line width=1pt](-4.2,1.5) - - (-5.7,0.75);
        \draw[line width=1pt](-4.2,0) - - (-5.7,0.75);
        \draw[line width=1pt](0.3,1.5) - - (0.3,0);
        \draw[line width=1pt](-2.7,1.5) - - (-2.7,0);
        \draw[line width=1pt](0.3,1.5) - - (-1.2,1.5);
        \draw[line width=1pt](0.3,0) - - (-2.7,0);
        \draw[line width=1pt](0.3,0) - - (1.8,0);
        \draw[line width=1pt](0.3,1.5) - - (1.8,1.5);
        \draw[line width=1pt](1.8,1.5) - - (1.8,0);
        \draw[line width=1pt](1.8,1.5) - - (3.3,0.75);
        \draw[line width=1pt](1.8,0) - - (3.3,0.75);
        \draw[line width=1pt](-1.2,1.5) - - (-2.7,0);
        \draw[line width=1pt](-1.2,1.5) - - (-2.7,1.5);
        \draw[line width=1pt](0.3,0) - - (1.8,1.5);
        \draw[line width=1pt](-1.2,0) - - (-1.2,1.5);
        \draw[line width=1pt](-1.2,0) - - (0.3,1.5);
        \draw[line width=1pt](3.3,0.75) to [bend right=45 ](-4.2,1.5);
        \draw[line width=1pt](-5.7,0.75) to [bend right=50 ](1.8,0);

		\node [draw, shape=circle,fill=white] (f) at  (-5.7,0.75) {};
	    \node [draw, shape=circle,fill=white] (a) at  (-4.2,1.5) {};
        \node [draw, shape=circle,fill=white] (b) at  (-2.7,1.5) {};
        \node [draw, shape=circle,fill=white] (c) at  (-1.2,1.5) {};
	    \node [draw, shape=circle,fill=white] (h) at  (-4.2,0) {};
        \node [draw, shape=circle,fill=white] (i) at  (-2.7,0) {};
        \node [draw, shape=circle,fill=white] (j) at  (-1.2,0) {};
	    \node [draw, shape=circle,fill=white] (d) at  (0.3,1.5) {};
        \node [draw, shape=circle,fill=white] (k) at  (0.3,0) {};
   		\node [draw, shape=circle,fill=white] (e) at  (1.8,1.5) {};
		\node [draw, shape=circle,fill=white] (l) at  (1.8,0) {};
		\node [draw, shape=circle,fill=white] (g) at  (3.3,0.75) {};

	\end{tikzpicture}
\qquad
    \centering
\begin{tikzpicture}[scale=.4, transform shape]
		\node [scale=1.3] at (-0.7,1.9) {$1$};
		\node [scale=1.3] at (-2.2,2) {$8$};
		\node [scale=1.3] at (0.8,2) {$2$};
		\node [scale=1.3] at (2.1,1.9) {$11$};
		\node [scale=1.3] at (2.2,-0.5) {$12$};
		\node [scale=1.3] at (4.2,0.75) {$13$};
		\node [scale=1.3] at (-3.7,1.9) {$6$};
        \node [scale=1.3] at (-0.7,-0.5) {$4$};
        \node [scale=1.3] at (-3.7,-0.5) {$7$};
        \node [scale=1.3] at (0.8,-0.5) {$10$};
        \node [scale=1.3] at (-2.2,-0.5) {$9$};
        \node [scale=1.3] at (-5.2,1.9) {$4$};
        \node [scale=1.3] at (-8.7,0.55) {$1$};
        \node [scale=1.3] at (-5.2,-0.5) {$5$};
        \node [scale=1.3] at (-6.4,-0.5) {$2$};
        \node [scale=1.3] at (-6.7,2) {$3$};
\node [scale=1.3] at (-7.7,-1.2) {$M_5$};
        \draw[line width=1pt](-8.2,0.75) - - (-6.7,0);
        \draw[line width=1pt](-8.2,0.75) - - (-6.7,1.5);
        \draw[line width=1pt](-6.7,1.5) - - (-5.2,1.5);
        \draw[line width=1pt](-6.7,0) - - (-5.2,0);
        \draw[line width=1pt](-5.2,1.5) - - (-3.7,1.5);
        \draw[line width=1pt](-5.2,0) - - (-3.7,0);
        \draw[line width=1pt](-3.7,1.5) - - (-2.2,1.5);
        \draw[line width=1pt](-3.7,0) - - (-2.2,0);
        \draw[line width=1pt](-2.2,1.5) - - (-0.7,1.5);
        \draw[line width=1pt](-2.2,0) - - (-0.7,0);
        \draw[line width=1pt](-0.7,1.5) - - (0.8,1.5);
        \draw[line width=1pt](-0.7,0) - - (0.8,0);
        \draw[line width=1pt](0.8,0) - - (2.2,0);
        \draw[line width=1pt](0.8,1.5) - - (2.2,1.5);
        \draw[line width=1pt](2.2,0) - - (3.7,0.75);
        \draw[line width=1pt](2.2,1.5) - - (3.7,0.75);
        \draw[line width=1pt](2.2,0) - - (2.2,1.5);
        \draw[line width=1pt](0.8,0) - - (0.8,1.5);
        \draw[line width=1pt](-0.7,0) - - (-0.7,1.5);
        \draw[line width=1pt](-2.2,0) - - (-2.2,1.5);
        \draw[line width=1pt](-3.7,0) - - (-3.7,1.5);
        \draw[line width=1pt](-5.2,0) - - (-5.2,1.5);
        \draw[line width=1pt](-6.7,0) - - (-6.7,1.5);
        \draw[line width=1pt](-6.7,0) - - (-5.2,1.5);
        \draw[line width=1pt](-5.2,0) - - (-3.7,1.5);
        \draw[line width=1pt](-3.7,0) - - (-2.2,1.5);
        \draw[line width=1pt](-2.2,0) - - (-0.7,1.5);
        \draw[line width=1pt](-0.7,0) - - (0.8,1.5);
        \draw[line width=1pt](0.8,0) - - (2.2,1.5);

        \draw[line width=1pt](-6.7,1.5) to [bend left=45 ](-3.7,1.5);
        \draw[line width=1pt](-2.2,1.5) to [bend left=50 ](0.8,1.5);
        \draw[line width=1pt](-3.7,0) to [bend right=40 ](0.8,0);
        \draw[line width=1pt](-5.2,0) to [bend left=45 ](-8.2,0.75);
        \draw[line width=1pt](-8.2,0.75) to [bend right=50 ](2.2,0);
        \draw[line width=1pt](-8.2,0.75) to [bend right=60 ](3.7,0.75);
        \draw[line width=1pt](-6.7,1.5) to [bend left=50 ](3.7,0.75);
	   \node [draw, shape=circle,fill=white] (h) at  (-8.2,0.75) {};
	    \node [draw, shape=circle,fill=white] (a) at  (-6.7,1.5) {};
	    \node [draw, shape=circle,fill=white] (b) at  (-5.2,1.5) {};
	    \node [draw, shape=circle,fill=white] (c) at  (-3.7,1.5) {};
	    \node [draw, shape=circle,fill=white] (d) at  (-2.2,1.5) {};
	    \node [draw, shape=circle,fill=white] (e) at  (-0.7,1.5) {};
	    \node [draw, shape=circle,fill=white] (f) at  (0.8,1.5) {};
	    \node [draw, shape=circle,fill=white] (j) at  (-6.7,0) {};
	    \node [draw, shape=circle,fill=white] (k) at  (-5.2,0) {};
	    \node [draw, shape=circle,fill=white] (l) at  (-3.7,0) {};
	    \node [draw, shape=circle,fill=white] (m) at  (-2.2,0) {};
	    \node [draw, shape=circle,fill=white] (n) at  (-0.7,0) {};
	    \node [draw, shape=circle,fill=white] (o) at  (0.8,0) {};
	    \node [draw, shape=circle,fill=white] (g) at  (2.2,1.5) {};
	    \node [draw, shape=circle,fill=white] (p) at  (2.2,0) {};
	    \node [draw, shape=circle,fill=white] (i) at  (3.7,0.75) {};

	\end{tikzpicture}
\qquad
    \centering
\begin{tikzpicture}[scale=.4, transform shape]
		\node [scale=1.3] at (-0.7,2) {$6$};
		\node [scale=1.3] at (-2.2,2) {$5$};
		\node [scale=1.3] at (1.1,1.2) {$7$};
		\node [scale=1.3] at (2.2,1.9) {$8$};
		\node [scale=1.3] at (2.2,-0.5) {$16$};
		\node [scale=1.3] at (4.2,0.75) {$9$};
		\node [scale=1.3] at (-3.7,2) {$4$};
        \node [scale=1.3] at (-1.1,.3) {$14$};
        \node [scale=1.3] at (-3.9,-0.4) {$12$};
        \node [scale=1.3] at (0.8,-0.5) {$15$};
        \node [scale=1.3] at (-2.2,-0.5) {$13$};
        \node [scale=1.3] at (-5.5,1.8) {$3$};
        \node [scale=1.3] at (-8.7,0.75) {$1$};
        \node [scale=1.3] at (-5.4,-0.4) {$11$};
        \node [scale=1.3] at (-7.1,.6) {$10$};
        \node [scale=1.3] at (-6.7,2) {$2$};
\node [scale=1.3] at (-7.7,-1.2) {$M_6$};
       \draw[line width=1pt](-8.2,0.75) - - (-6.7,0);
        \draw[line width=1pt](-8.2,0.75) - - (-6.7,1.5);
        \draw[line width=1pt](-6.7,1.5) - - (-5.2,1.5);
        \draw[line width=1pt](-6.7,0) - - (-5.2,0);
        \draw[line width=1pt](-5.2,1.5) - - (-3.7,1.5);
        \draw[line width=1pt](-5.2,0) - - (-3.7,0);
        \draw[line width=1pt](-3.7,1.5) - - (-2.2,1.5);
        \draw[line width=1pt](-3.7,0) - - (-2.2,0);
        \draw[line width=1pt](-2.2,1.5) - - (-0.7,1.5);
        \draw[line width=1pt](-2.2,0) - - (-0.7,0);
        \draw[line width=1pt](-0.7,1.5) - - (0.8,1.5);
        \draw[line width=1pt](-0.7,0) - - (0.8,0);
        \draw[line width=1pt](0.8,0) - - (2.2,0);
        \draw[line width=1pt](0.8,1.5) - - (2.2,1.5);
        \draw[line width=1pt](2.2,0) - - (3.7,0.75);
        \draw[line width=1pt](2.2,1.5) - - (3.7,0.75);
        \draw[line width=1pt](2.2,0) - - (2.2,1.5);
        \draw[line width=1pt](0.8,0) - - (0.8,1.5);
        \draw[line width=1pt](-0.7,0) - - (-0.7,1.5);
        \draw[line width=1pt](-2.2,0) - - (-2.2,1.5);
        \draw[line width=1pt](-3.7,0) - - (-3.7,1.5);
        \draw[line width=1pt](-5.2,0) - - (-5.2,1.5);
        \draw[line width=1pt](-6.7,0) - - (-6.7,1.5);
        \draw[line width=1pt](-6.7,0) - - (-5.2,1.5);
        \draw[line width=1pt](-5.2,0) - - (-3.7,1.5);
        \draw[line width=1pt](-3.7,0) - - (-2.2,1.5);
        \draw[line width=1pt](-2.2,0) - - (-0.7,1.5);
        \draw[line width=1pt](-0.7,0) - - (0.8,1.5);
        \draw[line width=1pt](0.8,0) - - (2.2,1.5);

        \draw[line width=1pt](-6.7,1.5) to [bend left=45 ](-3.7,1.5);
        \draw[line width=1pt](-2.2,1.5) to [bend left=45 ](2.2,1.5);
        \draw[line width=1pt](-0.7,1.5) to [bend left=45 ](2.2,1.5);
        \draw[line width=1pt](-2.2,0) to [bend right=40 ](0.8,0);
        \draw[line width=1pt](-2.2,0) to [bend left=40 ](-6.7,0);
        \draw[line width=1pt](-6.7,0) to [bend right=45 ](0.8,0);
        \draw[line width=1pt](-8.2,0.75) to [bend right=50 ](2.2,0);
        \draw[line width=1pt](-8.2,0.75) to [bend right=60 ](3.7,0.75);
        \draw[line width=1pt](-6.7,1.5) to [bend left=50 ](3.7,0.75);
	   \node [draw, shape=circle,fill=white] (h) at  (-8.2,0.75) {};
	    \node [draw, shape=circle,fill=white] (a) at  (-6.7,1.5) {};
	    \node [draw, shape=circle,fill=white] (b) at  (-5.2,1.5) {};
	    \node [draw, shape=circle,fill=white] (c) at  (-3.7,1.5) {};
	    \node [draw, shape=circle,fill=white] (d) at  (-2.2,1.5) {};
	    \node [draw, shape=circle,fill=white] (e) at  (-0.7,1.5) {};
	    \node [draw, shape=circle,fill=white] (f) at  (0.8,1.5) {};
	    \node [draw, shape=circle,fill=white] (j) at  (-6.7,0) {};
	    \node [draw, shape=circle,fill=white] (k) at  (-5.2,0) {};
	    \node [draw, shape=circle,fill=white] (l) at  (-3.7,0) {};
	    \node [draw, shape=circle,fill=white] (m) at  (-2.2,0) {};
	    \node [draw, shape=circle,fill=white] (n) at  (-0.7,0) {};
	    \node [draw, shape=circle,fill=white] (o) at  (0.8,0) {};
	    \node [draw, shape=circle,fill=white] (g) at  (2.2,1.5) {};
	    \node [draw, shape=circle,fill=white] (p) at  (2.2,0) {};
	    \node [draw, shape=circle,fill=white] (i) at  (3.7,0.75) {};

	\end{tikzpicture}
\qquad
    \centering
\begin{tikzpicture}[scale=.4, transform shape]
		\node [scale=1.3] at (-0.7,2) {$6$};
		\node [scale=1.3] at (-2.2,2) {$5$};
		\node [scale=1.3] at (1.1,1.2) {$7$};
		\node [scale=1.3] at (2.2,2) {$8$};
		\node [scale=1.3] at (1.8,0.4) {$16$};
		\node [scale=1.3] at (4.2,0.75) {$9$};
		\node [scale=1.3] at (-3.7,1.9) {$4$};
        \node [scale=1.3] at (-1.1,.4) {$14$};
        \node [scale=1.3] at (-4.1,0.4) {$12$};
        \node [scale=1.3] at (0.8,-0.5) {$15$};
        \node [scale=1.3] at (-2.2,-0.5) {$13$};
        \node [scale=1.3] at (-5.5,1.8) {$3$};
        \node [scale=1.3] at (-8.7,0.75) {$1$};
        \node [scale=1.3] at (-5.4,-0.4) {$11$};
        \node [scale=1.3] at (-7.1,.6) {$10$};
        \node [scale=1.3] at (-6.7,2) {$2$};
\node [scale=1.3] at (-7.7,-1.2) {$M_7$};
       \draw[line width=1pt](-8.2,0.75) - - (-6.7,0);
        \draw[line width=1pt](-8.2,0.75) - - (-6.7,1.5);
        \draw[line width=1pt](-6.7,1.5) - - (-5.2,1.5);
        \draw[line width=1pt](-6.7,0) - - (-5.2,0);
        \draw[line width=1pt](-5.2,1.5) - - (-3.7,1.5);
        \draw[line width=1pt](-5.2,0) - - (-3.7,0);
        \draw[line width=1pt](-3.7,1.5) - - (-2.2,1.5);
        \draw[line width=1pt](-3.7,0) - - (-2.2,0);
        \draw[line width=1pt](-2.2,1.5) - - (-0.7,1.5);
        \draw[line width=1pt](-2.2,0) - - (-0.7,0);
        \draw[line width=1pt](-0.7,1.5) - - (0.8,1.5);
        \draw[line width=1pt](-0.7,0) - - (0.8,0);
        \draw[line width=1pt](0.8,0) - - (2.2,0);
        \draw[line width=1pt](0.8,1.5) - - (2.2,1.5);
        \draw[line width=1pt](2.2,0) - - (3.7,0.75);
        \draw[line width=1pt](2.2,1.5) - - (3.7,0.75);
        \draw[line width=1pt](2.2,0) - - (2.2,1.5);
        \draw[line width=1pt](0.8,0) - - (0.8,1.5);
        \draw[line width=1pt](-0.7,0) - - (-0.7,1.5);
        \draw[line width=1pt](-2.2,0) - - (-2.2,1.5);
        \draw[line width=1pt](-3.7,0) - - (-3.7,1.5);
        \draw[line width=1pt](-5.2,0) - - (-5.2,1.5);
        \draw[line width=1pt](-6.7,0) - - (-6.7,1.5);
        \draw[line width=1pt](-6.7,0) - - (-5.2,1.5);
        \draw[line width=1pt](-5.2,0) - - (-3.7,1.5);
        \draw[line width=1pt](-3.7,0) - - (-2.2,1.5);
        \draw[line width=1pt](-2.2,0) - - (-0.7,1.5);
        \draw[line width=1pt](-0.7,0) - - (0.8,1.5);
        \draw[line width=1pt](0.8,0) - - (2.2,1.5);

        \draw[line width=1pt](-6.7,1.5) to [bend left=45 ](-3.7,1.5);
        \draw[line width=1pt](-2.2,1.5) to [bend left=45 ](2.2,1.5);
        \draw[line width=1pt](-0.7,1.5) to [bend left=45 ](2.2,1.5);
        \draw[line width=1pt](-3.7,1.5) to [bend left=45 ](2.2,1.5);
        \draw[line width=1pt](-2.2,0) to [bend right=40 ](0.8,0);
        \draw[line width=1pt](-2.2,0) to [bend left=40 ](-6.7,0);
        \draw[line width=1pt](-5.2,0) to [bend right=40 ](-2.2,0);
        \draw[line width=1pt](-8.2,0.75) to [bend right=40 ](2.2,0);
        \draw[line width=1pt](-8.2,0.75) to [bend right=50 ](3.7,0.75);
        \draw[line width=1pt](-6.7,1.5) to [bend left=55 ](3.7,0.75);
	   \node [draw, shape=circle,fill=white] (h) at  (-8.2,0.75) {};
	    \node [draw, shape=circle,fill=white] (a) at  (-6.7,1.5) {};
	    \node [draw, shape=circle,fill=white] (b) at  (-5.2,1.5) {};
	    \node [draw, shape=circle,fill=white] (c) at  (-3.7,1.5) {};
	    \node [draw, shape=circle,fill=white] (d) at  (-2.2,1.5) {};
	    \node [draw, shape=circle,fill=white] (e) at  (-0.7,1.5) {};
	    \node [draw, shape=circle,fill=white] (f) at  (0.8,1.5) {};
	    \node [draw, shape=circle,fill=white] (j) at  (-6.7,0) {};
	    \node [draw, shape=circle,fill=white] (k) at  (-5.2,0) {};
	    \node [draw, shape=circle,fill=white] (l) at  (-3.7,0) {};
	    \node [draw, shape=circle,fill=white] (m) at  (-2.2,0) {};
	    \node [draw, shape=circle,fill=white] (n) at  (-0.7,0) {};
	    \node [draw, shape=circle,fill=white] (o) at  (0.8,0) {};
	    \node [draw, shape=circle,fill=white] (g) at  (2.2,1.5) {};
	    \node [draw, shape=circle,fill=white] (p) at  (2.2,0) {};
	    \node [draw, shape=circle,fill=white] (i) at  (3.7,0.75) {};

	\end{tikzpicture}
\qquad
    \centering
\begin{tikzpicture}[scale=.4, transform shape]
		\node [scale=1.3] at (-0.7,2) {$5$};
		\node [scale=1.3] at (-2.2,2) {$4$};
        \node [scale=1.3] at (3.8,2) {$8$};
		\node [scale=1.3] at (3.1,.4) {$16$};
		\node [scale=1.3] at (1.2,1.2) {$6$};
		\node [scale=1.3] at (2.5,1.2) {$7$};
		\node [scale=1.3] at (1.8,0.4) {${15}$};
		\node [scale=1.3] at (5.7,0.75) {${9}$};
		\node [scale=1.3] at (-3.8,2) {$3$};
        \node [scale=1.3] at (-1.1,.4) {${13}$};
        \node [scale=1.3] at (-4.1,0.4) {${11}$};
        \node [scale=1.3] at (0.8,-0.5) {${14}$};
        \node [scale=1.3] at (-2.2,-0.5) {${12}$};
        \node [scale=1.3] at (-5.5,2) {$2$};
        \node [scale=1.3] at (-5.4,-0.4) {${10}$};
        \node [scale=1.3] at (-7.1,.6) {${1}$};
\node [scale=1.3] at (-7.7,-1.2) {$M_8$};
       \draw[line width=1pt](-6.7,0.75) - - (-5.2,0);
        \draw[line width=1pt](-6.7,0.75) - - (-5.2,1.5);
        \draw[line width=1pt](5.2,0.75) - - (3.7,0);
        \draw[line width=1pt](5.2,0.75) - - (3.7,1.5);
        \draw[line width=1pt](2.2,0) - - (3.7,1.5);
        \draw[line width=1pt](-5.2,1.5) - - (-3.7,1.5);
        \draw[line width=1pt](-5.2,0) - - (-3.7,0);
        \draw[line width=1pt](-3.7,1.5) - - (-2.2,1.5);
        \draw[line width=1pt](-3.7,0) - - (-2.2,0);
        \draw[line width=1pt](-2.2,1.5) - - (-0.7,1.5);
        \draw[line width=1pt](-2.2,0) - - (-0.7,0);
        \draw[line width=1pt](-0.7,1.5) - - (0.8,1.5);
        \draw[line width=1pt](-0.7,0) - - (0.8,0);
        \draw[line width=1pt](0.8,0) - - (2.2,0);
        \draw[line width=1pt](0.8,1.5) - - (2.2,1.5);
        \draw[line width=1pt](2.2,0) - - (3.7,0);
        \draw[line width=1pt](2.2,1.5) - - (3.7,1.5);
        \draw[line width=1pt](2.2,0) - - (2.2,1.5);
         \draw[line width=1pt](3.7,0) - - (3.7,1.5);
        \draw[line width=1pt](0.8,0) - - (0.8,1.5);
        \draw[line width=1pt](-0.7,0) - - (-0.7,1.5);
        \draw[line width=1pt](-2.2,0) - - (-2.2,1.5);
        \draw[line width=1pt](-3.7,0) - - (-3.7,1.5);
        \draw[line width=1pt](-5.2,0) - - (-5.2,1.5);
        \draw[line width=1pt](-5.2,0) - - (-3.7,1.5);
        \draw[line width=1pt](-3.7,0) - - (-2.2,1.5);
        \draw[line width=1pt](-2.2,0) - - (-0.7,1.5);
        \draw[line width=1pt](-0.7,0) - - (0.8,1.5);
        \draw[line width=1pt](0.8,0) - - (2.2,1.5);

        \draw[line width=1pt](.7,1.5) to [bend left=40 ](3.7,1.5);
        \draw[line width=1pt](-.8,1.5) to [bend left=45 ](3.7,1.5);
        \draw[line width=1pt](-2.3,1.5) to [bend left=50 ](3.7,1.5);
        \draw[line width=1pt](-3.8,1.5) to [bend left=55 ](3.7,1.5);
        \draw[line width=1pt](5.2,0.75) to [bend left=70 ](-6.8,0.75);
        \draw[line width=1pt](5.2,0.75) to [bend right=60 ](-5.3,1.5);
        \draw[line width=1pt](-5.2,0.) to [bend right=58 ](5.2,.75);

        \draw[line width=1pt](.7,0) to [bend right=40 ](3.7,0);
        \draw[line width=1pt](-.8,0) to [bend right=45 ](3.7,0);
        \draw[line width=1pt](-2.3,0) to [bend right=50 ](3.7,0);
        \draw[line width=1pt](-3.7,0) to [bend right=45 ](3.7,0);

	    \node [draw, shape=circle,fill=white] (a) at  (-6.7,.75) {};
	    \node [draw, shape=circle,fill=white] (b) at  (-5.2,1.5) {};
	    \node [draw, shape=circle,fill=white] (c) at  (-3.7,1.5) {};
	    \node [draw, shape=circle,fill=white] (d) at  (-2.2,1.5) {};
	    \node [draw, shape=circle,fill=white] (e) at  (-0.7,1.5) {};
	    \node [draw, shape=circle,fill=white] (f) at  (0.8,1.5) {};
        \node [draw, shape=circle,fill=white] (f) at  (3.7,1.5) {};
	    \node [draw, shape=circle,fill=white] (k) at  (-5.2,0) {};
	    \node [draw, shape=circle,fill=white] (l) at  (-3.7,0) {};
	    \node [draw, shape=circle,fill=white] (m) at  (-2.2,0) {};
	    \node [draw, shape=circle,fill=white] (n) at  (-0.7,0) {};
	    \node [draw, shape=circle,fill=white] (o) at  (0.8,0) {};
\node [draw, shape=circle,fill=white] (f) at  (3.7,0) {};
	    \node [draw, shape=circle,fill=white] (g) at  (2.2,1.5) {};
	    \node [draw, shape=circle,fill=white] (p) at  (2.2,0) {};
	    \node [draw, shape=circle,fill=white] (i) at  (5.2,0.75) {};
\end{tikzpicture}
\qquad
    \centering

    \begin{tikzpicture}[scale=.4, transform shape]
		\node [scale=1.3] at (-10.5,1) {$1$};
		\node [scale=1.3] at (-8.5,2.7) {$2$};
        \node [scale=1.3] at (8.5,2.7) {$n$};
		\node [scale=1.3] at (11,1) {$n+1$};
		\node [scale=1.3] at (-9,-.7) {$n+2$};
		\node [scale=1.3] at (8.5,-.7) {$2n$};
\node [scale=1.3] at (-11.3,-1.2) {$M_n$ ($n\ge 9$)};
       \draw[line width=1pt](-10,1) - - (-8,2.5);
        \draw[line width=1pt](-10,1) - - (-8,-.5);

        \draw[line width=1pt](10,1) - - (8,2.5);
        \draw[line width=1pt](10,1) - - (8,-.5);

        \draw[line width=1pt](-8,2.5) - - (-6,2.5);
        \draw[line width=1pt](-6,2.5) - - (-4,2.5);
        \draw[line width=1pt](-4,2.5) - - (-2,2.5);
        \draw[line width=1pt](-2,2.5) - - (-1.5,2.5);

        \draw[line width=1pt](1.5,2.5) - - (2,2.5);
        \draw[line width=1pt](2,2.5) - - (4,2.5);
        \draw[line width=1pt](4,2.5) - - (6,2.5);
        \draw[line width=1pt](6,2.5) - - (8,2.5);

        \draw[line width=1pt](0,2.5) - - (.5,2.5);
        \draw[line width=1pt](0,2.5) - - (-.5,2.5);
        \draw[line width=1pt](0,-.5) - - (.5,-.5);
        \draw[line width=1pt](0,-.5) - - (-.5,-.5);

        \draw[line width=1pt](-8,-.5) - - (-6,-.5);
        \draw[line width=1pt](-6,-.5) - - (-4,-.5);
        \draw[line width=1pt](-4,-.5) - - (-2,-.5);
        \draw[line width=1pt](-2,-.5) - - (-1.5,-.5);

        \draw[line width=1pt](1.5,-.5) - - (2,-.5);
        \draw[line width=1pt](2,-.5) - - (4,-.5);
        \draw[line width=1pt](4,-.5) - - (6,-.5);
        \draw[line width=1pt](6,-.5) - - (8,-.5);

        \draw[line width=1pt](-8,2.5) - - (-8,-.5);
        \draw[line width=1pt](-6,2.5) - - (-6,-.5);
        \draw[line width=1pt](-4,2.5) - - (-4,-.5);
        \draw[line width=1pt](-2,2.5) - - (-2,-.5);

        \draw[line width=1pt](0,2.5) - - (0,-.5);

        \draw[line width=1pt](2,2.5) - - (2,-.5);
        \draw[line width=1pt](4,2.5) - - (4,-.5);
        \draw[line width=1pt](6,2.5) - - (6,-.5);
        \draw[line width=1pt](8,2.5) - - (8,-.5);

        \draw[line width=1pt](8,2.5) - - (6,-.5);
        \draw[line width=1pt](6,2.5) - - (4,-.5);
        \draw[line width=1pt](4,2.5) - - (2,-.5);
        \draw[line width=1pt](2,2.5) - - (1.6,2);

        \draw[line width=1pt](0,2.5) - - (-.6,2);
        \draw[line width=1pt](0,-.5) - - (.6,0);
        \draw[line width=1pt](-2,-.5) - - (-1.4,0);

        \draw[line width=1pt](-2,2.5) - - (-4,-.5);
        \draw[line width=1pt](-4,2.5) - - (-6,-.5);
        \draw[line width=1pt](-6,2.5) - - (-8,-.5);

       \draw[line width=1pt](8,2.5) to [bend right=40 ](4,2.5);
        \draw[line width=1pt](8,2.5) to [bend right=45 ](2,2.5);

        \node [scale=1.3] at (3,4) {$\ddots $};
        \node [scale=1.3] at (2.5,3.7) {$\ddots $};
        \node [scale=1.3] at (.5,4) {$\ddots $};
        \node [scale=1.3] at (1,4.3) {$\ddots $};

        \node [scale=1.3] at (1.2,-1) {$\cdots $};
        \node [scale=1.3] at (1.9,-1.5) {$\cdots $};
        \node [scale=1.3] at (-.9,-1) {$\cdots $};
        \node [scale=1.3] at (-.3,-1.5) {$\cdots $};

        \draw[line width=1pt](8,2.5) to [bend right=50 ](0,2.5);
        \draw[line width=1pt](8,2.5) to [bend right=55 ](-2,2.5);
        \draw[line width=1pt](8,2.5) to [bend right=60 ](-4,2.5);
        \draw[line width=1pt](8,2.5) to [bend right=65 ](-6,2.5);
        \draw[line width=1pt](10,1) to [bend right=65 ](-8,2.5);

        \draw[line width=1pt](8,-.5) to [bend left=40 ](4,-.5);
        \draw[line width=1pt](8,-.5) to [bend left=45 ](2,-.5);
        \draw[line width=1pt](8,-.5) to [bend left=50 ](0,-.5);
        \draw[line width=1pt](8,-.5) to [bend left=55 ](-2,-.5);
        \draw[line width=1pt](8,-.5) to [bend left=60 ](-4,-.5);
        \draw[line width=1pt](8,-.5) to [bend left=65 ](-6,-.5);
        \draw[line width=1pt](10,1) to [bend left=65 ](-8,-.5);
        \draw[line width=1pt](10,1) to [bend left=95 ](-10,1);

	   \node [draw, shape=circle,fill=white] (h) at  (-10,1) {};
	    \node [draw, shape=circle,fill=white] (a) at  (-8,2.5) {};
	    \node [draw, shape=circle,fill=white] (b) at  (-6,2.5) {};
	    \node [draw, shape=circle,fill=white] (c) at  (-4,2.5) {};
        \node [draw, shape=circle,fill=white] (d) at  (-2,2.5) {};
\node [scale=1.3] at (-1,2.5) {$\dots$};
	    \node [draw, shape=circle,fill=white] (e) at  (0,2.5) {};
\node [scale=1.3] at (1,2.5) {$\dots$};
        \node [draw, shape=circle,fill=white] (e) at  (0,-.5) {};
	    \node [draw, shape=circle,fill=white] (f) at  (2,2.5) {};
        \node [draw, shape=circle,fill=white] (t) at  (4,2.5) {};
        \node [draw, shape=circle,fill=white] (s) at  (6,2.5) {};
        \node [draw, shape=circle,fill=white] (r) at  (8,2.5) {};

	    \node [draw, shape=circle,fill=white] (h) at  (10,1) {};
	    \node [draw, shape=circle,fill=white] (a) at  (-8,-.5) {};
	    \node [draw, shape=circle,fill=white] (b) at  (-6,-.5) {};
	    \node [draw, shape=circle,fill=white] (c) at  (-4,-.5) {};
        \node [draw, shape=circle,fill=white] (d) at  (-2,-.5) {};
\node [scale=1.3] at (-1,-.5) {$\dots$};
	    \node [draw, shape=circle,fill=white] (e) at  (0,2.5) {};
\node [scale=1.3] at (1,-.5) {$\dots$};
	    \node [draw, shape=circle,fill=white] (f) at  (2,-.5) {};
        \node [draw, shape=circle,fill=white] (t) at  (4,-.5) {};
        \node [draw, shape=circle,fill=white] (s) at  (6,-.5) {};
        \node [draw, shape=circle,fill=white] (r) at  (8,-.5) {};
	\end{tikzpicture}
\caption{The graphs $G$  related to Problem 1.3 for $\Delta \ge 3$.}\label{Fig-problem-n}
\end{figure}

For $\Delta(G)\ge 9$ consider, the graph $M_n$  ($n\ge 9$) of order $2n$, Figure \ref{Fig-problem-n}, which is seen $\Delta(G)=n$ and $\chi_{ei}(M_n)=2n> \lfloor \frac{3n}{2} \rfloor+1 =\lfloor \frac{3\Delta}{2} \rfloor+1$.

With this regard, the Problem \ref{prob-conj} is disproved. In the other words the type of Conjecture 2 for $e$-injective coloring is rejected. However, from the Observations \ref{obs-e-inj} and \ref{obs-e-dis}, we can have.

\begin{corollary}
Let $G$ be a graph with the property that the given data in Observation, \ref{obs-e-dis} (Conversely part) hold. Then the Conjecture 1 is wrong.\\
Let $G$ be a graph with the property that the given data in Observation, \ref{obs-e-inj} (Conversely part) hold. Then the Conjecture 2 is wrong.\\
\end{corollary}

For tree $T$ we have.
\begin{proposition}\label{prop equ-T}
Let $T$ be a tree. Then

\emph{1}. $\chi_{ei}(T)=1$  if and only if $diam(T)\le 2$.

\emph{2}. $\chi_{ei}(T)=2$  if and only if $diam(T)\ge 3$.
\end{proposition}

\begin{proof}
1. If $diam(T)=1$, then $T=P_2$  and if $diam(T)=2$, then $T$ is a star and since there is no path of length $3$ between any two vertices,
$\chi_{ei}(T)=1$.\\
Conversely, let $\chi_{ei}(T)=1$. Then there is no path of length $3$ in $T$. Thus $diam(T)\le 2$.

2.  Let $diam(T)\ge 3$ and $v_0$ be a vertex of maximum degree in $T$. We assign color $1$ to the $v_0$ and to the vertex $u$ if $d(u,v_0)$ is even, and  color $2$ to the vertex $u$ if $d(u,v_0)$ is odd. Since there is  only one path between any two vertices in any tree $T$, so
if two vertices $x,y$ are in distance $3$ and two vertices $x,z$ are in distance $3$, then two vertices $y,z$ are not in distance $3$. This shows that, we can use color $1$ for $x$ and color $2$ for $y,z$. Therefore $\chi_{ei}(T)\le 2$. On the other hand,
if $diam(T)\ge 3$, then $\chi_{ei}(T)\ge 2$. Therefore, if $diam(T)\ge 3$, then $\chi_{ei}(T)=2$.

If $\chi_{ei}(T)= 2$, there is two vertices $v,w$ in $T$ so that the path $vxyw$ is of length $3$ in $T$.
This shows that $diam(T)\ge 3$.
\end{proof}

For graphs $G$ and $H$, let $G\cup H$ be the disjoint union of $G$ and $H$. Then it is easy to see that $\chi_{ei}(G\cup H) = max\{\chi_{ei}(G),
\chi_{ei}(H)\}$.\\
For the join of two graphs, we have the following.

\begin{proposition}\label{prop vest}
Let $G$ and $H$ be two graphs of order $m$ and $n$ respectively, with the property that, $E(G)$ and $E(H)$ are non-empty sets. Then  $\chi_{ei}(G\vee H)=m+n$
\end{proposition}

\begin{proof}
Let $e_1=v_1w_1\in E(G)$ and $e_2=v_2w_2\in E(H)$ be  two edges. We show that any two vertices $x, y$ in $G\vee H$, there is a path of length $3$, with end vertices $x,y$. For observing the result,
we bring up five positions.

1. For $x, y\in V(G)$, consider the path $xv_2w_2y$ in $G\vee H$.

2. For  $x, y\in V(H)$, consider the path $xv_1w_1y$ in $G\vee H$.

3. For  $x\in V(G)\setminus \{v_1,w_1\}$ and $y\in V(H)\setminus \{v_2,w_2\}$, consider the path $xv_2v_1y$ in $G\vee H$.

4. For $x\in \{v_1,w_1\}$, say $v_1$ and $y\in V(H)\setminus \{v_2,w_2\}$,  consider the path $v_1v_2w_1y$ in $G\vee H$.

5. For $x\in \{v_1,w_1\}$ and $y\in\{v_2,w_2\}$ and without loss of generality, say $x=v_1$ and $y=v_2$, consider the path $v_1w_1w_2v_2$ in $G\vee H$.\\ The other positions are similar.
Therefore, for any two vertices $x,y \in G\vee H$ there is a path of length $3$, with end vertices $x,y$.
Therefore the result is observed.
\end{proof}

Let $G$ be a graph and $B$ be a maximum packing set of $G$. If $v\in V(G)\setminus B$, then there is a vertex $u\in B$ such that
$N(v)\cap N(u)\ne \emptyset$. This shows that, $d(v,u)\le 2$. Thus $B$ is a $2$-distance dominating set. Therefore we have.

\begin{proposition}\label{prop pack}
Let $G$ be a graph of diameter $3$. Then $\chi_{ei}(G)\ge \rho(G)\ge \gamma_{2}(G)$. One can have the equalities.
\end{proposition}

\begin{proof} Let $B$ be maximum packing set of graph $G$. Since two vertices of $B$ has distance $3$, they are assigned with two
distinct  colors. Thus $\chi_{ei}(G)\ge \rho$.\\
For equalities, consider the cycles $C_6$ and $C_8$, (see Propositions \ref{prop-cn}).
\end{proof}

We now give an upper bound on $\chi_{ei}(G)$ that may be  slightly important.

\begin{proposition}\label{Max deg}
 Let $G$ have maximum degree $\Delta$. Then, $\chi_{ei}(G)\le \Delta(\Delta-1)^2+ 1$. This bound is sharp for odd cycle $C_n$
 $(n\ge 5)$.
\end{proposition}

\begin{proof} Let $G$ be a graph and $v \in V(S_3(G))=V(G)$. It is well known that there are at most $\Delta(\Delta-1)^2$ vertices in $G$ such that any of them with $v$ form two end vertices of path $P_4$. This shows that $\deg_{S_{3(G)}}(v)\le \Delta(\Delta-1)^2$. On the other hand $\chi_{ei}(G)= \chi(S_3(G))$ and from Brooks Theorem in usual coloring of graphs,
$\chi(S_3(G))\le \Delta(S_3(G))+1\le \Delta(\Delta-1)^2+1$. Therefore $\chi_{ei}(G)\le \Delta(\Delta-1)^2+1$.\\
For seeing the sharpness observe Proposition \ref{prop-cn}.
\end{proof}

Also we want to drive bounds for the $e$-injective coloring of Cartesian product of two graphs $G$, $H$ in terms of $2$-distance coloring of the
of $G$ and $H$. For this we explore a result from \cite{Mojdeh} and a lemma.
\begin{theorem}\emph{(\cite{Mojdeh} Theorem 1)}\label{theo1-moj}
For any graphs $G$ and $H$ with no isolated vertices,
$$(\Delta(G)+1)(\Delta(H)+1)\leq \chi_{2}(G\boxtimes H)\leq \chi_{2}(G)\chi_{2}(H).$$
\end{theorem}

\begin{lemma}\label{lem-G*H}
Let $G$ and $H$ be two graphs with no isolated vertices. If  two end vertices of each path $P_4$ in $G$ and $H$ are adjacent or have a common neighbor, then so does $G\boxtimes H$.
\end{lemma}

\begin{proof}
Suppose that the end vertices of each path $P_4$ in graphs $G$ and $H$ are adjacent or have a common neighbor. We would to be show any two end vertices of a path $P_4$ in $G\boxtimes H$ are adjacent or have a common neighbor. For this, we can bring up  the possible paths $P_4$  in graph $G\boxtimes H$.\\

1.1. $(a,u)(a,v)(a,w)(a,t)$;\ \ 1.2. $(a,u)(a,v)(a,w)(b,w)$;\ \ 1.3.  $(a,u)(a,v)(a,w)(b,t)$.

2.1. $(a,u)(a,v)(b,v)(b,w)$;\ \ 2.2. $(a,u)(a,v)(b,v)(c,v)$;\ \ 2.3 $(a,u)(a,v)(b,v)(c,w)$.

3.1. $(a,u)(a,v)(b,w)(b,t)$;\ \ 3.2. $(a,u)(a,v)(b,w)(c,w)$;\ \ 3.3. $(a,u)(a,v)(b,w)(c,t)$.

4.1. $(a,u)(b,u)(b,v)(b,w)$;\ \ 4.2. $(a,u)(b,u)(b,v)(c,v)$;\ \ 4.3. $(a,u)(b,u)(b,v)(c,w)$.

5.1. $(a,u)(b,u)(c,u)(c,v)$;\ \ 5.2. $(a,u)(b,u)(c,u)(d,u)$;\ \ 5.3. $(a,u)(b,u)(c,u)(d,v)$.

6.1. $(a,u)(b,u)(c,v)(c,w)$;\ \ 6.2. $(a,u)(b,u)(c,v)(d,v)$;\ \ 6.3.  $(a,u)(b,u)(c,v)(d,w)$.

7.1.  $(a,u)(b,v)(b,w)(b,t)$;\ \ 7.2. $(a,u)(b,v)(b,w)(c,w)$;\ \ 7.3. $(a,u)(b,v)(b,w)(c,t)$.

8.1. $(a,u)(b,v)(c,v)(c,w)$;\ \ 8.2. $(a,u)(b,v)(c,v)(d,v)$;\ \ 8.3. $(a,u)(b,v)(c,v)(d,w)$.

9.1. $(a,u)(b,v)(c,w)(d,w)$;\ \ 9.2. $(a,u)(b,v)(c,w)(c,t)$;\ \ 9.3. $(a,u)(b,v)(c,w)(d,t)$.\\

Now we observe that, all these paths type $P_4$ are adjacent or have a common neighbor.\\

1.1. Since $uvwt$ is a path $P_4$ in $H$, the vertices $u$ and $t$ are adjacent or have a common neighbor. If $u$ and $t$ are adjacent, then
the vertices $(a,u)$ and $(a,t)$ are adjacent in $G\boxtimes H$. If  $u$ and $t$ have a common neighbor $s$, then $(a,s)$ is a common neighbor of
$(a,u)$ and $(a,t)$.\\
1.2. $(a,v)$ is a common neighbor of $(a,u)$ and $(b,w)$ in $G\boxtimes H$.\\
1.3. The $uvwt$ is a path $P_4$ in $H$. If $u$ and $t$ are adjacent, then
the vertices $(a,u)$ and $(b,t)$ are adjacent in $G\boxtimes H$. If  $u$ and $t$ have a common neighbor $s$, then $(a,s)$ is a common neighbor of
$(a,u)$ and $(b,t)$.

2.1. The vertex $(b,v)$ is a common neighbor of $(a,u)$ and $(b,w)$ in $G\boxtimes H$.\\
2.2. The vertex $(b,v)$ is a common neighbor of $(a,u)$ and $(c,v)$ in $G\boxtimes H$.\\
2.3. The vertex $(b,v)$ is a common neighbor of $(a,u)$ and $(c,w)$ in $G\boxtimes H$.

3.1. Its proof is readily and similar to the proof of 1.3.\\
3.2. The vertex $(b,v)$ is a common neighbor of $(a,u)$ and $(c,w)$ in $G\boxtimes H$.\\
3.3. Its proof is readily, and is similar to the proof of 1.3.

4.1. The vertex $(b,v)$ is a common neighbor of $(a,u)$ and $(b,w)$ in $G\boxtimes H$.\\
4.2. The vertex $(b,v)$ is a common neighbor of $(a,u)$ and $(c,v)$ in $G\boxtimes H$.\\
4.3. The vertex $(b,v)$ is a common neighbor of $(a,u)$ and $(c,w)$ in $G\boxtimes H$.

5.1. The vertex $(b,u)$ is a common neighbor of $(a,u)$ and $(c,v)$ in $G\boxtimes H$.\\
5.2. Since $abcd$ is a path $P_4$ in $G$, the vertices $a$ and $d$ are adjacent or have a common neighbor. If $a$ and $d$ are adjacent, then
the vertices $(a,u)$ and $(d,u)$ are adjacent in $G\boxtimes H$. If  $a$ and $d$ have a common neighbor $r$, then $(r,u)$ is a common neighbor of
$(a,u)$ and $(d,u)$.\\
5.3. Its proof is obvious and it is similar to the proof of 5.2.

6.1. The vertex $(b,v)$ is a common neighbor of $(a,u)$ and $(c,v)$ in $G\boxtimes H$.\\
6.2. Its proof is obvious and it is similar to the proof of 5.2.\\
6.3. Its proof is obvious and it is similar to the proof of 5.2.

7.1. Its proof is  similar to the proof of 1.3.\\
7.2. The vertex $(b,v)$ is a common neighbor of $(a,u)$ and $(c,w)$ in $G\boxtimes H$.\\
7.3. Its proof is  similar to the proof of 1.3.

8.1. The vertex $(b,v)$ is a common neighbor of $(a,u)$ and $(c,w)$ in $G\boxtimes H$.\\
8.2. Its proof is  similar to the proof of 5.2.\\
8.3. Its proof is  similar to the proof of 5.2.

9.1. Its proof is similar to the proof of 5.2.\\
9.2. Its proof is  similar to the proof of 1.3.\\
9.3. There are two paths $abcd$ and $uvwt$ in $G$ and $H$ respectively. If $ad\in E(G)$  and $ut\in E(H)$, then $(a,d)$ and $(u,t)$ are adjacent in
$G\boxtimes H$. If $ad\in E(G)$ and $s$ is a common neighbor of $u$ and $t$, then $(a,s)$ is a common neighbor of $(a,u)$ and $(d,t)$ in
$G\boxtimes H$. If $ut\in E(H)$ and $r$ is a common neighbor of $a$ and $d$, then $(r,u)$ is a common neighbor of $(a,u)$ and $(d,t)$ in
$G\boxtimes H$. If $a$ and $d$ have a common neighbor $r$,  and similarly, $s$ is a common neighbor of $u$ and $t$, then $(r,s)$ is a common
neighbor of $(a,u)$ and $(d,t)$ in $G\boxtimes H$.\\
It is observed that, both end vertices of every path $P_4$ in $G\boxtimes H$ are adjacent or have a common neighbor. Therefore the proof is complete.
\end{proof}

Now we have the following.
\begin{theorem} \label{theo-G*H}
	For any graphs $G$ and $H$  with no isolated vertices, with the property that, any two end vertices of each path $P_4$ in $G$ and $H$ are adjacent or have a common neighbor, we have
$$ Max\{\chi_{ei}(G), \chi_{ei}(H)\} \leq \chi_{ei}(G\square H)\leq  \chi_{2}(G)\chi_{2}(H).$$ The bounds are sharp.
\end{theorem}
\begin{proof}
For the first inequality, since $G$ and $H$ have no isolated vertices, and any path of length $3$ of $G$ and $H$ gives at least one path of length $3$ of $G\square H$, thus the
first inequality holds. For seeing the sharpness, consider $G=P_m$ and $H=P_n$ where $m\ge 4$ or $n\ge 4$.

We now prove the second inequality. From  the definitions of Cartesian and strong products, we may have $G\square H$ as a subgraph of  $G\boxtimes H$, and next any path $P_4$ of $(G\square H)$ is a path $P_4$ of $(G\boxtimes H)$.
Therefore, $\chi_{ei}(G\square H)\leq \chi_{ei}(G\boxtimes H)$. As the same way, $\chi_{2}(G\square H)\leq \chi_{2}(G\boxtimes H)$.
From Observation \ref{obs-e-dis} and Lemma \ref{lem-G*H} $\chi_{ei}(G\boxtimes H) \le \chi_{2}(G\boxtimes H)$.
On the other hand, from Theorem \ref{theo1-moj}, $\chi_{2}(G\boxtimes H)\leq \chi_{2}(G)\chi_{2}(H)$. These deduce  that $\chi_{ei}(G\square H)\leq \chi_{2}(G)\chi_{2}(H)$. It is easy to see that, this bound is sharp for $G=C_3$ and $H=C_5$ and also $G=C_3$ and $H=C_7$,
(see Theorem \ref{theo-C*C*}).
\end{proof}

\section{$e$-injective chromatic number of some graphs}
In this section we investigate the $e$-injective coloring of some custom graphs.
Now we determine exact value for the $e$-injective chromatic number for some special families of graphs, such as path, cycle, complete graphs, wheel graphs, star,  complete bipartite graphs, $k$-regular bipartite graphs, multipartite graphs,  fan graphs,  prism graphs, ladder graphs and hypercube graphs, double star and trees.

\begin{proposition}\label{prop-pn}
For Path $P_{n}$, we have
\end{proposition}
$$\chi_{ei}(P_n)=\begin{cases}
 1 &\ \  n\leq3 \\
 2 &\  \ n>3 \quad\\
 \end{cases}.$$
\begin{proof}  Let $n\leq3$. It is obvious that $\chi_{ei}(P_n)=1$.
    \\ Let $n>3$. We assign the color $1$ to the odd vertices and  color $2$ to the even vertices. These assignment shows that, every pair of vertices which there exists a path of length $3$ between them, receive distinct colors. Therefore $\chi_{ei}(P_{n})=2$.
\end{proof}

\begin{proposition}\label{prop-cn}
For cycle $C_{n}$, we have
$$\chi_{ei}(C_n)=\begin{cases}
 1 & n=3 \\
 2 & n\ge 4,\ 2|n \\
 3 & n\ge 4,\ 2\nmid n \\
 \end{cases}.$$
 \end{proposition}
\begin{proof}
 Let $n=3$. It is obvious that $\chi_{ei}(C_3)=1$. \\
 Let $n>3$. There are two cases to be considered.
 
     \textbf{Case $1$}. If $n$ is even.\\
  We assign the color $1$ to the odd vertices and  color $2$ to the even vertices. Therefore $\chi_{ei}(C_{2k})=2$.
  
     \textbf{Case $2$}. If $n$ is odd.\\
      Let $n=5$. We assign the color $1$ to the vertices $v_{1}, v_{2}$ and color $2$ to the vertices $v_{3},v_{4}$ and we assign color $3$ to the vertex $v_{5}$. This assignments is an $e$-injective coloring of $C_{5}$.\\
      Let $n\geq 7$  . We assign the color $1$ to the odd vertices $v_i$s,  for $i\le n-4$ and color $2$ to the even vertices $v_i$s,  for $i\le n-3$ and we assign color $3$ to the vertices $v_{n-2}, v_{n-1}, v_n$. This assignments is an $e$-injective coloring of $C_n$ for odd $n\geq 7 $. On the other hand,
since there are two paths of length $3$ between $v_{n-2}$ with two vertices $v_{1}$, $v_{n-5}$, as well as $v_{n-1}$ with two vertices $v_{2}$, $v_{n-4}$ and also $v_{n}$ with two vertices $v_{3}$, $v_{n-3}$. It is clear that, one cannot $e$-injective color to the vertices cycle $C_n$ withe two colors for odd $n$. Therefore the result holds.
\end{proof}

Since for $n\ge 4$, any two vertices are end of a path $P_4$, then we have.
\begin{observation}\label{obs-kn}
For complete graph $K_{n}$, we have
$$\chi_{ei}(K_n)=\begin{cases}
 1 &\  \ n\leq3 \\
 n &\  \ n>3 \quad \\
 \end{cases}.$$
 \end{observation}

\begin{proposition}\label{prop-wn}
For wheel graph $W_{n} (n\ge 3),\  \chi_{ei}(W_n)=n+1$.
\end{proposition}
\begin{proof}
Let $v_{1}$ be a universal vertex. For $i,j\ge 2$, there exists a path $v_jv_{j+1}v_1v_i$ between $v_i$ and $v_j$ if $v_i\ne v_{j+1}$; and
there exists a path $v_jv_{1}v_{j+2}v_i$ between $v_i$ and $v_j$ if $v_i= v_{j+1}$. On the other hand, there exists a path $v_1v_{i-2}v_{i-1}v_i$
between $v_1$ and $v_i$.
Taking this account, there exist a path of length $3$  between two vertices in $W_{n}$. Therefore $\chi_{ei}(W_{n})=n+1$.
\end{proof}

\begin{proposition}\label{prop-k{n,m}}
For complete bipartite graph $K_{n,m}$ with $m,n \ge 2$, $\chi_{ei}(K_{n,m})=2$.
\end{proposition}
\begin{proof}
Let $n\geq2,\ m\geq2$. It is easy to observe that, there is a path of length $3$ between two vertices of two  different partite sets. Therefore one can assign color $1$ to one partite set and $2$ to another partite set. Thus, $\chi_{ei}(K_{n,m})=2$.
\end{proof}
Using the proof of  proposition \ref{prop-k{n,m}}, For regular bipartite graphs, we have the next result, which proof is similar.

\begin{proposition}\label{prop-kr}
For $k$-regular bipartite graph  $G$ ($k\ge 2$),
$\chi_{ei}(G)=2$
 \end{proposition}

A complete $r$ partite graph is a simple graph such that the vertices are partitioned to $r$ independent vertex sets
and every pair of vertices are adjacent if and only if they belong to different partite sets.
\begin{proposition}\label{prop-rmul}
Let $G=K_{n_1,...,n_r}$ be a complete $(r\ge 3)$ partite graph of order $n$. Then

$$\chi_{ei}(G)=\begin{cases}
 1 &\  \ r=3\ \emph{and}\ G=K_{1,1,1}\\
 n-1 &\  \ r=3\ \emph{and}\ G\in \{K_{n-2,1,1}, K_{{1},n-2,1}, K_{{1},1, n-2}\}\ \emph{with}\ n\ge 4\\
 n &\  \ ow
 \end{cases}.$$
\end{proposition}

\begin{proof} 1. It is trivial.

2. Let $r=3$ and $G\in \{K_{n-2,1,1}, K_{{1},n-2,1}, K_{{1},1, n-2}\}$. Without loss of generality, assume that $G=K_{n-2,1,1}$ with vertex set
$V=\{v_1,v_2,\dots, v_{n-2}, u_1, w_1\}$. The path $v_iu_1w_1v_j$ is a path of length $3$ between $v_i, v_j$.
The paths $u_1v_iw_1v_j$ and $w_1v_iu_1v_j$ are  paths of length $3$ between $u_1, v_j$ and between $w_1, v_j$ respectively. On the other hand, there exists no path of length $3$ from  $u_1$ and $w_1$. Now we can assign same color to $u_1$, $w_1$ and $n-2$ other
different colors to the $v_i$s. Thus $\chi_{ei}(G)=n-1$.

3. If $r\ge 4$, and $v_i, w_j$ are two vertices of two partite sets, then taking two vertices from two other partite sets $x_m, y_l$,
one can construct a path $v_i, x_m, y_l, w_j$ of length $3$.\\
Let $r=3$ and $G=K_{k,l,m}$ where  two of $k,l,m$ are at least $2$. If $k=1$ and $l,m\ge 2$
with $V(G)=\{v_1, u_1\dots, u_l, w_1,\dots, w_m\}$,
 then $v_1u_i w_s  u_j$,
 $v_1w_i u_s  w_j$, $u_iv_1u_jw_s$, $u_iw_tv_1u_j$, $w_iu_xv_1w_j$  with $i\ne j$, give us a path of length $3$ between $v_1, u_j$, $v_1, w_j$, $u_i,w_s$, $u_i,u_j$ and $w_i,w_j$ respectively.\\
Let $r=3$ and $G=K_{k,l,m}$ where  $k,l,m\ge 2$. Then, similar to the  second part of situation 3, there exist a path of length $3$
between any two vertices of $G$. Therefore $\chi_{ei}(G)=n$. Thus the result holds.
\end{proof}

\begin{proposition}\label{prop-f{m,n}}
For fan graph $F_{m,n}$, we have
\end{proposition}
$$\chi_{ei}(F_{m,n})=\begin{cases}
 1 &\ \ m=1,n=2  \\
 m+1 &\ \ m\geq 2, n=2 \\
 m+n &\ \ m= 1, n\geq 4  \ \textmd{and} \ m\geq 2, n\geq 3\\
 \end{cases}.$$
\begin{proof}
 There are three situations to be considered.\\
    1. If $m=1,n=2$. It is obvious.\\
    2. If $m\geq 2, n=2$ and $V(F_{m,2})=\{v_1, v_2,\dots, v_m, u_1,u_2\}$, then by definition $F_{2,2}\cong F_{1,3}$
 and $F_{m,2}=\overline{K}_{m}\vee P_{2}$. Two vertices $u_1,u_2$ receive  same color because there is no  path of length $3$ between them. On the other hand there exist a path
  $v_iu_1u_2v_j$ of length $3$ between any pair of vertices $v_i,v_j$,  and there exist a path $v_iu_lv_ju_k$ with $l\neq k$ of length $3$ between any pair of vertices $v_i,u_k$.  Therefore $\chi_{ei}(F_{m,2})=m+1$.\\
3.1.  Let $m= 1, n\geq 4$  and $V(F_{1,n})=\{v_1, u_1,u_2,\dots, u_n\}$. Then $v_1u_{i+2}u_{i+1}u_{i}$, ($i\leq n-2$), $v_1u_{i-2}u_{i-1}u_{i}$, ($i\geq n-1$), $u_{i}v_1u_{j+1}u_{j}$, ($i<j<n$),
$u_{i}v_1u_{n-1}u_{n}$ ($i<n-1$) and $u_{n-1}u_{n-2}v_1u_{n}$
are the paths of length $3$ between any pair of vertices $u_i$, $u_j$ and $v_1$, $u_i$. Thus $\chi_{ei}(F_{1,n})=1+n$.\\
3.2. Let  $m\geq 2, n\geq 3$.
Using the reasons given in proof of part 3.1, one can easy to see that there is a path of length $3$ between any pair of vertices of $F_{m,n}$.\\
All in all the proof is completed.
\end{proof}

\section{Grid graphs, cylinder graphs and tori graphs}
In this section we argue on the $e$-injective coloring of three families of Cartesian product graphs (grid, cylinder and tori graphs),  and we gain the $e$-injective chromatic number of them. For this sake we partition this section to three subsections.
\subsection{Grid graphs}
In this subsection, our main purpose is to give the $e$-injective chromatic number for grid graphs. 
A {\it ladder graph} $L_n$ is a simple connected undirected grid graph $P_2\square P_n$ with $2n$ vertices and $3n-2$ edges.

\begin{proposition}\label{prop-ln}
Let $G\cong L_{n}$ be a ladder graph. Then $\chi_{ei}(L_{n})=2,\ (n\geq{2})$.
\end{proposition}
\begin{proof}
Let $V(L_n)=\{v_{1,1}, v_{1,2}, \dots, v_{1,n}, v_{2,1}, v_{2,2}, \dots, v_{2,n}\}$.  Since $L_{n}$ contains copy of $C_{4}$ as subgraph, so $\chi_{ei}(L_{n})\geq2$.\\
We show an $e$-injective coloring with $2$ colors for $L_n$.
For this, we define a function $f: V \rightarrow \{1,2\}$ as follows:\\
$$f(v_{i,j})=\begin{cases}
 1 & i=1,\  2\nmid j\ \emph{or}\ i=2,\ 2\mid j\\
 2 & i=1,\ 2\mid j\ \emph{or}\ i=2,\ 2\nmid j\\
 \end{cases}$$

This coloring  is easily seen that, give us  an $e$-injective coloring. Therefore, $\chi_{ei}(L_{n})=2$, (see Figure \ref{Fig4}).
\begin{figure}[h]
	\centering
	\begin{tikzpicture}[scale=.5, transform shape]

		\node [scale=1.8] at (-5,2.6) {$v_{1,1}$};
		\node [scale=1.8] at (-3,2.6) {$v_{1,2}$};
		\node [scale=1.8] at (-1.2,2.6) {$v_{1,3}$};
        \node [scale=1.8] at (1,2.6) {$v_{1,4}$};
		\node [scale=1.8] at (3,2.6) {$v_{1,5}$};

        \node [scale=1.8] at (-5,-0.8) {$v_{2,1}$};
		\node [scale=1.8] at (-3,-0.8) {$v_{2,2}$};
		\node [scale=1.8] at (-1,-0.8) {$v_{2,3}$};
		\node [scale=1.8] at (1,-0.8) {$v_{2,4}$};
        \node [scale=1.8] at (3,-0.8) {$v_{2,5}$};
        \draw[line width=1pt](-5,2) - - (3,2);
        \draw[line width=1pt](-5,0) - - (3,0);
        \draw[line width=1pt](-5,2) - - (-5,0);
        \draw[line width=1pt](-3,2) - - (-3,0);
        \draw[line width=1pt](-1,2) - - (-1,0);
        \draw[line width=1pt](1,2) - - (1,0);
        \draw[line width=1pt](3,2) - - (3,0);
        \draw[dotted] [line width=1pt](3,2) - - (9,2);
        \draw[dotted] [line width=1pt](3,0) - - (9,0);

        \node [draw, shape=circle,fill=red] (v1) at  (-5,2) {};
		\node [draw, shape=circle,fill=yellow] (v2) at  (-3,2) {};
		\node [draw, shape=circle,fill=red] (v3) at  (-1,2) {};
		\node [draw, shape=circle,fill=yellow] (v4) at  (1,2) {};
		\node [draw, shape=circle,fill=red] (v5) at  (3,2) {};

        \node [draw, shape=circle,fill=yellow] (u1) at  (-5,0) {};
        \node [draw, shape=circle,fill=red] (u2) at  (-3,0) {};
        \node [draw, shape=circle,fill=yellow] (u3) at  (-1,0) {};
        \node [draw, shape=circle,fill=red] (u4) at  (1,0) {};
        \node [draw, shape=circle,fill=yellow] (u5) at  (3,0) {};
	\end{tikzpicture}
\caption{the ladder graph of $L_n$}\label{Fig4}
\end{figure}
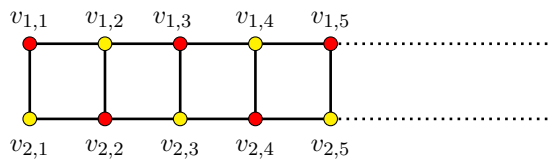

\end{proof}

\begin{theorem}\label{theo-P*P}
For paths $P_{m}$ and $P_{n}$ $(m, n\ge 2)$, we have $\chi_{ei}(P_{m}\square P_{n})=2$.
\end{theorem}
\begin{proof}
Let $V(P_{n}^{i})=\{v_{i,1},v_{i,2},v_{i,3},\ldots,v_{i,n}\}$ be the set vertex of $i$th row of $P_{m}\square P_{n}$. Since $v_{i,j},v_{i+1,j},v_{i+1,j+1},v_{i,j+1}$ is a path of length $3$ between any two adjacent vertices in $P_{n}$. Therefore $\chi_{ei}({P_{m}\square P_{n}})\geq 2$. \\
We construct an optimal $e$-injective coloring with $2$ colors of ${P_{m}\square P_{n}}$ using of pattern given in the matrix form in Table \ref{T1}. In this matrix {\it ij}th element represents the color of ${v_{i,j}}$. We define a function $f:V\rightarrow\{1,2\}$ as follows.\\
$$f(v_{i,j})=\begin{cases}
 1 & 2|i+j \\
 2 & 2\nmid i+j\\
 \end{cases}$$\\
Since there is no path of length $3$ between every pair of vertices with the same color, this gives an $e$-injective coloring of $({P_{m}\square P_{n}})$ with $2$ colors. Therefore $\chi_{ei}({P_{m}\square P_{n}})=2$.
\begin{table}[h]
  \centering
  $\begin{pmatrix}
    1 & 2 & 1 & \dots \\
    2 & 1 & 2 & \dots \\
    1 & 2 & 1 & \dots \\
    \vdots & \vdots & \vdots \\
  \end{pmatrix}$
  \caption{Matrix $e$-injective coloring of ${P_{m}\square P_{n}}$ }\label{T1}
\end{table}

\end{proof}

\subsection{Cylinder graphs}
In this subsection, our main purpose is to give the $e$-injective chromatic number for cylinder graphs. 

A {\it Prism graph} ( or circular ladder graph) is a simple cubic cylinder graph  of an $n$-gonal prism has $2n$ vertices and $3n$ edges and denoted
 by $D_{n}$, see Figure \ref{Fig-dn}.
 \begin{proposition}\label{prop-dn}
Let $G\cong D_n$ be a prism graph. Then
\end{proposition}
\begin{equation}
\chi_{ei}(D_n)=\begin{cases}
 6 & \textmd{if} \ n=3 \\
 5 & \textmd{if} \ n=5 \\
 4 & \textmd{if} \ n=7 \\
 2 & \textmd{if} \ n\ \emph{is}\ \emph{even}\ \emph{and}\ n\geq 4 \\
 3 & \textmd{if} \ n\ \emph{is}\ \emph{odd}\ \emph{and}\ n\geq 9 \\
 \end{cases}
\end{equation}
\begin{proof}
Let $V(D_{n})=\{v_{1},v_{2},\dots ,v_{n}, u_{1},u_{2},\dots ,u_{n}\}$. From now on suppose that the vertices $v_i$s are for outer cycle and the vertices $u_i$s are for inner cycle. We bring up  five situations.\\
1. Let $n=6$. It is  obvious that $\chi_{ei}(D_3)=6$.\\
2. Let $n=5$. There is a path  of length $3$ between any two different vertices in outer cycle $C_{5}$, and inner cycle $C_{5}$. So $\chi_{ei}(D_{5})\geq 5$. On the other hand, there is no path of length $3$ between $v_i$ and $u_{i-1}$ (mod $5$). This shows that, $v_i$ and $u_{i-1}$ can be assigned same color. Therefore $\chi_{ei}(D_{5})=5$. (See  $D_5$ in Figure \ref{Fig-dn})\\
3.   Let $n=7$. There is a path of length $3$ between any two adjacent vertices in inner cycle $C_{7}$ and also from Proposition \ref{prop-cn}, we realize that $\chi_{ei}(D_{7})\geq4$. On the other hand, there is no path of length $3$ between $v_i$ and $u_{i-1}$ (mod $7$). This shows that, $v_i$ and $u_{i-1}$ can be assigned  the same color. Therefore $\chi_{ei}(D_{7})=4$. \\
 (See $D_7$ in Figure \ref{Fig-dn})

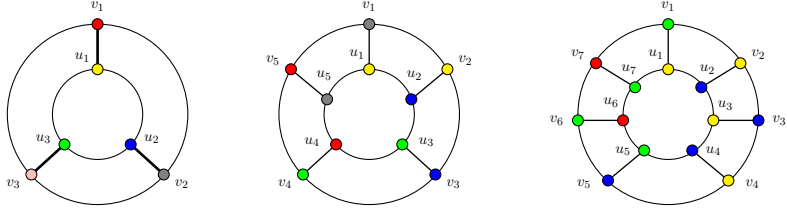
\begin{figure}[h]
	\centering
		\begin{tikzpicture}[scale=.4, transform shape]
		\node [scale=1.3] at (-0.5,1.9) {$u_{1}$};
		\node [scale=1.3] at (0,3.6) {$v_{1}$};
        \node [scale=1.3] at (1.8,-0.8) {$u_{2}$};
        \node [scale=1.3] at (2.8,-2.3) {$v_{2}$};
        \node [scale=1.3] at (-1.8,-0.8) {$u_{3}$};
		\node [scale=1.3] at (-2.8,-2.3) {$v_{3}$};
        \draw  (0,0) circle (1.5cm);
        \draw  (0,0) circle (3cm);
        \draw[line width=1pt](0,1.5) - - (0,3);
        \draw[line width=1pt](1.1,-1) - - (2.2,-2);
        \draw[line width=1pt](-1.1,-1) - - (-2.2,-2);
         \node [draw, shape=circle,fill=yellow] (v1) at  (0,1.5) {};
	    \node [draw, shape=circle,fill=blue] (v2) at  (1.1,-1) {};
        \node [draw, shape=circle,fill=green] (v3) at  (-1.1,-1) {};

		\node [draw, shape=circle,fill=red] (u1) at  (0,3) {};
	    \node [draw, shape=circle,fill=gray] (u2) at  (2.2,-2) {};
		\node [draw, shape=circle,fill=pink] (u3) at  (-2.2,-2) {};

	\end{tikzpicture}
\qquad
	\centering
	\begin{tikzpicture}[scale=.4, transform shape]

		\node [scale=1.3] at (-0.4,2) {$u_{1}$};
		\node [scale=1.3] at (0,3.6) {$v_{1}$};
		\node [scale=1.3] at (1.5,1.3) {$u_{2}$};
        \node [scale=1.3] at (3.2,1.8) {$v_{2}$};
        \node [scale=1.3] at (1.9,-0.9) {$u_{3}$};
        \node [scale=1.3] at (2.8,-2.3) {$v_{3}$};
        \node [scale=1.3] at (-1.9,-0.9) {$u_{4}$};
		\node [scale=1.3] at (-2.8,-2.3) {$v_{4}$};
        \node [scale=1.3] at (-1.5,1.3) {$u_{5}$};
		\node [scale=1.3] at (-3.2,1.8) {$v_{5}$};
        \draw  (0,0) circle (1.5cm);
        \draw  (0,0) circle (3cm);
        \draw[line width=0.5pt](0,1.5) - - (0,3);
        \draw[line width=0.5pt](1.1,-1) - - (2.2,-2);
        \draw[line width=0.5pt](1.4,0.5) - - (2.6,1.5);
        \draw[line width=0.5pt](-1.1,-1) - - (-2.2,-2);
        \draw[line width=0.5pt](-1.4,0.5) - - (-2.6,1.5);

        \node [draw, shape=circle,fill=yellow] (v3) at  (0,1.5) {};
	    \node [draw, shape=circle,fill=green] (v1) at  (1.1,-1) {};
		\node [draw, shape=circle,fill=blue] (v2) at  (1.4,0.5) {};
        \node [draw, shape=circle,fill=red] (v4) at  (-1.1,-1) {};
        \node [draw, shape=circle,fill=gray] (v5) at  (-1.4,0.5) {};

		\node [draw, shape=circle,fill=gray] (u1) at  (0,3) {};
		\node [draw, shape=circle,fill=yellow] (u2) at  (2.6,1.5) {};
	    \node [draw, shape=circle,fill=blue] (u3) at  (2.2,-2) {};
		\node [draw, shape=circle,fill=green] (u4) at  (-2.2,-2) {};
        \node [draw, shape=circle,fill=red] (u5) at  (-2.6,1.5) {};
	\end{tikzpicture}
\qquad
    \centering
	\begin{tikzpicture}[scale=.4, transform shape]

		\node [scale=1.3] at (-0.4,2) {$u_{1}$};
		\node [scale=1.3] at (0,3.6) {$v_{1}$};
		\node [scale=1.3] at (1.3,1.5) {$u_{2}$};
        \node [scale=1.3] at (3,2) {$v_{2}$};
        \node [scale=1.3] at (1.9,0.3) {$u_{3}$};
        \node [scale=1.3] at (2.8,-2.3) {$v_{4}$};
        \node [scale=1.3] at (-1.9,0.4) {$u_{6}$};
		\node [scale=1.3] at (-2.8,-2.3) {$v_{5}$};
        \node [scale=1.3] at (-1.5,-1.2) {$u_{5}$};
		\node [scale=1.3] at (-3,2) {$v_{7}$};
		\node [scale=1.3] at (-1.3,1.5) {$u_{7}$};
		\node [scale=1.3] at (1.5,-1.2) {$u_{4}$};
		\node [scale=1.3] at (3.7,-0.2) {$v_{3}$};
		\node [scale=1.3] at (-3.7,-0.2) {$v_{6}$};

        \draw  (0,0) circle (1.5cm);
        \draw  (0,0) circle (3cm);
        \draw[line width=0.5pt](0,1.5) - - (0,3);
        \draw[line width=0.5pt](1.1,0.9) - - (2.4,1.7);
        \draw[line width=0.5pt](-1.1,0.9) - - (-2.4,1.7);
        \draw[line width=0.5pt](1.5,-0.2) - - (3,-0.2);
        \draw[line width=0.5pt](-1.5,-0.2) - - (-3,-0.2);
        \draw[line width=0.5pt](-0.8,-1.2) - - (-2,-2.2);
        \draw[line width=0.5pt](0.8,-1.2) - - (2,-2.2);

        \node [draw, shape=circle,fill=yellow] (u1) at  (0,1.5) {};
		\node [draw, shape=circle,fill=blue] (u2) at  (1.1,0.9) {};
		\node [draw, shape=circle,fill=yellow] (u3) at  (1.5,-0.2) {};
	    \node [draw, shape=circle,fill=blue] (u4) at  (0.8,-1.2) {};
        \node [draw, shape=circle,fill=green] (u5) at  (-0.8,-1.2) {};
		\node [draw, shape=circle,fill=red] (u6) at  (-1.5,-0.2) {};
        \node [draw, shape=circle,fill=green] (u7) at  (-1.1,0.9) {};

		\node [draw, shape=circle,fill=green] (v1) at  (0,3) {};
		\node [draw, shape=circle,fill=yellow] (v2) at  (2.4,1.7) {};
        \node [draw, shape=circle,fill=green] (v6) at  (-3,-0.2) {};
        \node [draw, shape=circle,fill=blue] (v3) at  (3,-0.2) {};
	    \node [draw, shape=circle,fill=yellow] (v4) at  (2,-2.2) {};
		\node [draw, shape=circle,fill=blue] (v5) at  (-2,-2.2) {};
        \node [draw, shape=circle,fill=red] (v7) at  (-2.4,1.7) {};
	\end{tikzpicture}
\caption{{The prism graph of $D_n$ with $n\in \{3,  5,7\}$}}\label{Fig-dn}
\end{figure}
4.  Assume that $n\ge 4$ is even. Since $D_{n}$ contains a copy of $C_{n}$ as subgraph, we need at least two colors. This shows that $\chi_{ei}(D_n)\ge 2$. On the other hand, $v_iu_iu_{i+1}v_{i+1}$, $u_iv_iv_{i+1}u_{i+1}$, $v_iv_{i+1}u_{i+1}u_i$  are paths of length $3$ between $v_i,v_{i+1}$ (mod $n$) and between $u_i,u_{i+1}$ (mod $n$) and between $u_i,v_{i}$ respectively. It is straightforward to see that There is no path of length
$3$ between $v_i,u_{i+1}$ (mod $n$) and between $u_i,v_{i+1}$ (mod $n$).  Therefore  assigning colors $1$ to the vertices  $v_{2i+1}$ and $u_{2i}$ and
assigning color $2$  to the vertices $u_{2i+1}$ and $v_{2i}$ give us an $e$-injective coloring.
 Therefore $\chi_{ei}(D_{n})=2$, (see Figure \ref{Fig-even}).\\
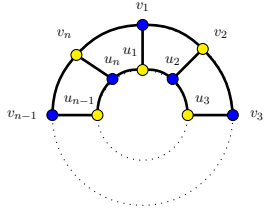
\begin{figure}[h]
	\centering
	\begin{tikzpicture}[scale=.4, transform shape]

\node [scale=1.3] at (-2.1,0.5) {$u_{n-1}$};
		\node [scale=1.3] at (-1,1.8) {$u_{n}$};
		\node [scale=1.3] at (-0.4,2.1) {$u_{1}$};
		\node [scale=1.3] at (1,1.8) {$u_{2}$};
		\node [scale=1.3] at (2,0.5) {$u_{3}$};
		\node [scale=1.3] at (-4,0) {$v_{n-1}$};
		\node [scale=1.3] at (-2.6,2.7) {$v_{n}$};
		\node [scale=1.3] at (0,3.6) {$v_{1}$};
        \node [scale=1.3] at (2.6,2.7) {$v_{2}$};
        \node [scale=1.3] at (3.8,0) {$v_{3}$};
        \draw [dotted] (0,0) circle (1.5cm);
        \draw [dotted] (0,0) circle (3cm);
        \draw[line width=1pt](-3,0) to [bend left=20 ](-2.2,2);
        \draw[line width=1pt](-2.2,2) to [bend left=23 ](0,3);
        \draw[line width=1pt](0,3) to [bend left=23 ](2,2.2);
        \draw[line width=1pt](2,2.2) to [bend left=20 ](3,0);

        \draw[line width=1pt](-1.5,0) to [bend left=23 ](-1,1.2);
        \draw[line width=1pt](-1,1.2) to [bend left=23 ](0,1.5);
        \draw[line width=1pt](0,1.5) to [bend left=23 ](1,1.2);
        \draw[line width=1pt](1,1.2) to [bend left=20 ](1.5,0);

        \draw[line width=1pt](-3,0) - - (-1.5,0);
        \draw[line width=1pt](-2.2,2) - - (-1,1.2);
        \draw[line width=1pt](0,3) - - (0,1.5);
        \draw[line width=1pt](2,2.2) - - (1,1.2);
        \draw[line width=1pt](3,0) - - (1.5,0);
        \node [draw, shape=circle,fill=blue] (v1) at  (-3,0) {};
        \node [draw, shape=circle,fill=yellow] (v2) at  (-2.2,2) {};
		\node [draw, shape=circle,fill=blue] (v3) at  (0,3) {};
		\node [draw, shape=circle,fill=yellow] (v4) at  (2,2.2) {};
	    \node [draw, shape=circle,fill=blue] (v5) at  (3,0) {};
		
  \node [draw, shape=circle,fill=yellow] (v1) at  (-1.5,0) {};
        \node [draw, shape=circle,fill=blue] (v2) at  (-1,1.2) {};
		\node [draw, shape=circle,fill=yellow] (v3) at  (0,1.5) {};
		\node [draw, shape=circle,fill=blue] (v4) at  (1,1.2) {};
	    \node [draw, shape=circle,fill=yellow] (v5) at  (1.5,0) {};
	\end{tikzpicture}
\caption{The prism graph of $D_n$ for even  $n$}\label{Fig-even}
\end{figure}
5. Assume that $n\ge 9$ is odd, (as depicted in Figure  \ref{Fig-odd}).
Since $n$ is odd. It is obvious that, $\chi_{ei}(D_{n})\ge 3$ (see Figure  \ref{Fig-odd}). Since $n$ is odd, then we can consider three moods for $n$.\\
5.1 Let $3\mid n$ and $n=3k$. Then, we assign color $1$ to any vertex in  $\{v_1, u_2, v_3\}\cup \{v_{6j-1}, u_{6j-2}, u_{6j}: 1\le j\le \frac{3k-3}{6}\}$, assign color $2$ to any vertex in $\{v_2,u_1,u_3\}\cup  \{v_{6j+1}, u_{6j+2}, v_{6j+3}: 1\le j\le \frac{3k-3}{6}\}$, and assign color $3$ to
any vertex in $\{\{v_{2i+2}, u_{2i+3}: 1\le i\le \frac{3k-3}{2}\}$.\\
5.2 Let $3\mid n-2$ and $n=3k+2$. Then, the vertices $v_i$ and $u_i$ for $1\le i\le 3k$ take same colors used in part 5.1, and next the vertices $v_{3k+1}, u_{3k+2}$ can be assigned by color $1$, and  the vertices $u_{3k+1}, v_{3k+2}$ can be assigned by color $2$.\\
5.3 Let $3\mid n-4$ and $n=3k+4$. Then, the vertices $v_i$ and $u_i$ for $1\le i\le 3k$ take same colors used in part 5.1, and next the vertices
$v_{3k+1}, v_{3k+3}, u_{3k+2} u_{3k+4}$ can be assigned by color $1$, and  the vertices $u_{3k+1}, v_{3k+2} u_{3k+3}$ can be assigned by color $2$
and finally, assign color $3$ for $v_{3k+4}$. These show $\chi_{ei}(D_{n})\le 3$.  Therefore we deduce that $\chi_{ei}(D_{n})= 3$ while $n$ is odd.

    Therefore $\chi_{ei}(D_{n})=3$.\\
 \begin{figure}[h]
	\centering
	\begin{tikzpicture}[scale=.4, transform shape]

		\node [scale=1.3] at (0,3.6) {$v_{1}$};
        \node [scale=1.3] at (1.8,3.2) {$v_{2}$};
        \node [scale=1.3] at (3,2.1) {$v_{3}$};
        \node [scale=1.3] at (3.6,1) {$v_{4}$};
        \node [scale=1.3] at (3.7,0) {$v_{5}$};
        \node [scale=1.3] at (3.6,-1) {$v_{6}$};
        \node [scale=1.3] at (3,-2.1) {$v_{7}$};
        \node [scale=1.3] at (1.8,-3.2) {$v_{8}$};
        \node [scale=1.3] at (0,-3.6) {$v_{9}$};
		\node [scale=1.3] at (-1.8,3.2) {$v_{n}$};
		\node [scale=1.3] at (-3.1,2.1) {$v_{n-1}$};

		\node [scale=1.3] at (0,1.4) {$u_{1}$};
		\node [scale=1.3] at (0.6,1.4) {$u_{2}$};
		\node [scale=1.3] at (0,-1.4) {$u_{9}$};
		\node [scale=1.3] at (-0.6,1.4) {$u_{n}$};

        \draw [dotted] (0,0) circle (2cm);
        \draw [dotted] (0,0) circle (3cm);
        \draw[line width=1pt](3,0) to [bend left=20 ](2.8,-1);
        \draw[line width=1pt](0,3) to [bend left=23 ](2,2.2);
        \draw[line width=1pt](2,2.2) to [bend left=20 ](3,0);
        \draw[line width=1pt](2.8,-1) to [bend left=20 ](2.2,-2);
        \draw[line width=1pt](2.2,-2) to [bend left=20 ](1.3,-2.7);
        \draw[line width=1pt](1.3,-2.7) to [bend left=18 ](0,-3);
        \draw[line width=1pt](0,-3) to [bend left=20 ](-1.3,-2.7);
        \draw[line width=1pt](-1.3,-2.7) to [bend left=20 ](-2.2,-2);
        \draw[line width=1pt](-2.3,1.8) to [bend left=25 ](0,3);

        \draw[line width=1pt](0,2) to [bend left=23 ](0.8,1.9);
        \draw[line width=1pt](0.8,1.9) to [bend left=23 ](1.5,1.5);
        \draw[line width=1pt](1.5,1.5) to [bend left=20 ](1.9,0.8);
        \draw[line width=1pt](1.9,0.8) to [bend left=20 ](2,0);
        \draw[line width=1pt](2,0) to [bend left=20 ](1.9,-0.8);
        \draw[line width=1pt](1.9,-0.8) to [bend left=20 ](1.5,-1.5);
        \draw[line width=1pt](1.5,-1.5) to [bend left=20 ](0.8,-1.9);
        \draw[line width=1pt](0.8,-1.9) to [bend left=20 ](0,-2);
        \draw[line width=1pt](0,-2) to [bend left=20 ](-0.8,-1.9);
        \draw[line width=1pt](-0.8,-1.9) to [bend left=20 ](-1.5,-1.5);
        \draw[line width=1pt](-0.8,1.9) to [bend left=20 ](-1.5,1.5);
        \draw[line width=1pt](0,2) to [bend left=20 ](-0.8,1.9);

        \draw[line width=1pt](0,2) - - (0,3);
        \draw[line width=1pt](0.8,1.9) - - (1.5,2.6);
        \draw[line width=1pt](1.5,1.5) - - (2.3,1.8);
        \draw[line width=1pt](1.9,0.8) - - (2.8,1);
        \draw[line width=1pt](2,0) - - (3,0);
        \draw[line width=1pt](1.9,-0.8) - - (2.8,-1);
        \draw[line width=1pt](1.5,-1.5) - - (2.2,-2);
        \draw[line width=1pt](0.8,-1.9) - - (1.3,-2.7);
        \draw[line width=1pt](0,-2) - - (0,-3);
        \draw[line width=1pt](-0.8,-1.9) - - (-1.3,-2.7);
        \draw[line width=1pt](-1.5,-1.5) - - (-2.2,-2);
        \draw[line width=1pt](-1.5,1.5) - - (-2.3,1.8);
        \draw[line width=1pt](-0.8,1.9) - - (-1.5,2.6);

		\node [draw, shape=circle,fill=green] (u1) at  (0,2) {};
        \node [draw, shape=circle,fill=red] (u2) at  (0.8,1.9) {};
		\node [draw, shape=circle,fill=green] (u3) at  (1.5,1.5) {};
        \node [draw, shape=circle,fill=red] (u4) at  (1.9,0.8) {};
        \node [draw, shape=circle,fill=pink] (u5) at  (2,0) {};
        \node [draw, shape=circle,fill=red] (u6) at  (1.9,-0.8) {};
        \node [draw, shape=circle,fill=pink] (u7) at  (1.5,-1.5) {};
        \node [draw, shape=circle,fill=green] (u8) at  (0.8,-1.9) {};
        \node [draw, shape=circle,fill=pink] (u9) at  (0,-2) {};
        \node [draw, shape=circle,fill=white] (u10) at  (-0.8,-1.9) {};
	    \node [draw, shape=circle,fill=white] (u11) at  (-1.5,-1.5) {};
        \node [draw, shape=circle,fill=white] (un-1) at  (-1.5,1.5) {};
        \node [draw, shape=circle,fill=white] (un) at  (-0.8,1.9) {};

        \node [draw, shape=circle,fill=red] (v1) at  (0,3) {};
        \node [draw, shape=circle,fill=green] (v2) at  (1.5,2.6) {};
        \node [draw, shape=circle,fill=red] (v3) at  (2.3,1.8) {};
        \node [draw, shape=circle,fill=pink] (v4) at  (2.8,1) {};
        \node [draw, shape=circle,fill=red] (v5) at  (3,0) {};
        \node [draw, shape=circle,fill=pink] (v6) at  (2.8,-1) {};
        \node [draw, shape=circle,fill=green] (v7) at  (2.2,-2) {};
        \node [draw, shape=circle,fill=pink] (v8) at  (1.3,-2.7) {};
        \node [draw, shape=circle,fill=green] (v9) at  (0,-3) {};
        \node [draw, shape=circle,fill=white] (v10) at  (-1.3,-2.7) {};
        \node [draw, shape=circle,fill=white] (v11) at  (-2.2,-2) {};
        \node [draw, shape=circle,fill=white] (vn-1) at  (-2.3,1.8) {};
		\node [draw, shape=circle,fill=white] (vn) at  (-1.5,2.6) {};

	\end{tikzpicture}
\caption{The prism graph of $D_n$ for odd $n\geq9$}\label{Fig-odd}
\end{figure}
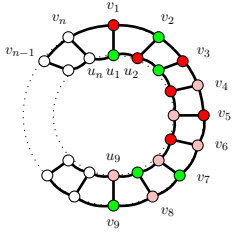

\end{proof}

\begin{theorem}\label{theo-P*C}
For any path $P_{m}$ and any cycle $C_{n}$, we have
\end{theorem}
$$\chi_{ei}({P_{m}\square C_{n}})=\begin{cases}
 6 & m\geq2, \ n=3 \\
 5 & m\geq2, \ n=5 \\
 4 & m\geq2, \ n=7 \\
 2 & m\geq2, \  2|n  \\
 3 & m\geq2, \  2\nmid n, \ (n\geq9) \\
\end{cases}.$$
\begin{proof}
Let $V(C_{n}^{i})=\{v_{i,1},v_{i,2},v_{i,3},\ldots,v_{i,n}\}$. There are four cases to be considered.

\textbf{Case 1}. Let $m\geq2, \ n=3$. There is a path of length $3$ between all vertices in $C_{3}^{i}$ and $C_{3}^{i+1}$. So $\chi_{ei}({P_{m}\square C_{3}})\geq6$.\\
 We construct an optimal $e$-injective coloring with $6$ colors of ${P_{m}\square C_{3}}$ using of pattern given in  Table \ref{T2}. For this, we define the function $f:V\rightarrow\{1,2,3,4,5,6\}$ as follows.\\
 $$f(v_{i,j})=\begin{cases}
 j & 2\nmid i \\
 3+j & 2|i \\
 \end{cases}.$$\\
It is obvious that, we need at  $6$ different colors for all vertices in $C_{3}^{i}$ and $C_{3}^{i+1}$. On the other hand, the vertex $v_{i+2,j}$
can be assigned with the color of $v_{i,j}$ since there is no path of length $3$ between them.
Therefore, it is observed that $\chi_{ei}({P_{m}\square C_{3}})=6$.
\begin{table}[h]
  \centering
  $\begin{pmatrix}
    1 & 2 & 3 \\
    4 & 5 & 6 \\
    1 & 2 & 3 \\
    4 & 5 & 6 \\
    \vdots & \vdots & \vdots \\
  \end{pmatrix}$
  \caption{Matrix $e$-injective coloring of ${P_{m}\square C_{3}}$  }\label{T2}
\end{table}

\textbf{Case 2}. Let $m\geq2, \ n=5$. According to the Cartesian product, there is a path of length $3$ between any two different vertices in $C_{5}^{i}$. So $\chi_{ei}({P_{m}\square C_{5}})\geq5$.\\
We now show, there is an $e$-injective coloring with $5$ colors for ${P_{m}\square C_{5}}$ such as the Table \ref{T3}. For this, we define a function $f:V \rightarrow \{1,2,3,4,5\}$ as $f(v_{i,j})=f(v_{i-1,j-1\ (\emph{mod}\ 5)})$ for $i\geq2$.
\begin{table}[h]
  \centering
  $\begin{pmatrix}
    1 & 2 & 3 & 4 & 5 \\
    5 & 1 & 2 & 3 & 4 \\
    4 & 5 & 1 & 2 & 3 \\
    3 & 4 & 5 & 1 & 2 \\
    2 & 3 & 4 & 5 & 1 \\
    \vdots & \vdots &\vdots & \vdots& \vdots
  \end{pmatrix}$
  \caption{Matrix $e$-injective coloring of ${P_{m}\square C_{5}}$  }\label{T3}
\end{table}

Since for any two successive cycles $C_{n}^{i}$ and $C_{n}^{i+1}$ of ${P_{m}\square C_{5}}$, there is not a path of length $3$ between $(v_{i,j})$ and $(v_{i+1,j+1\ (\emph{mod}\ 5)})$, we may assign same color to these two vertices.
Therefore, it is observed that $$\chi_{ei}({P_{m}\square C_{5}})=5.$$\\

\begin{table}[h]
  \centering
  $\begin{pmatrix}
    1 & 2 & 1 & 2 & 3 & 4 & 3 \\
    3 & 1 & 2 & 1 & 2 & 3 & 4 \\
    4 & 3 & 1 & 2 & 1 & 2 & 3 \\
    3 & 4 & 3 & 1 & 2 & 1 & 2 \\
    2 & 3 & 4 & 3 & 1 & 2 & 1 \\
    \vdots & \vdots&\vdots&\vdots&\vdots&\vdots&\vdots
  \end{pmatrix}$
  \caption{Matrix $e$-injective coloring of ${P_{m}\square C_{7}}$  }\label{T4}
\end{table}

\textbf{Case 3}.  Let $m\geq2, \ n=7$. The Cartesian product $P_{m}\square C_{7}$ contains copy of $D_{7}$ as a subgraph, so $\chi_{ei}({P_{m}\square C_{7}})\geq4$.\\
We now show that, there is an $e$-injective coloring with $4$ colors for ${P_{m}\square C_{7}}$ such as the Table \ref{T4} which defines a function $f:V \rightarrow \{1,2,3,4\}$ as follows.\\
 $$f(v_{1,j})=\begin{cases}
 1 & \textmd{if} \ j\in \{1, 3\}\\
 2 & \textmd{if} \ j\in \{2, 4\}\\
 3 & \textmd{if} \ j\in\{5,7\} \\
 4 & \textmd{if} \ j=6 \\
 \end{cases}$$\\
                        $$ and $$
 $$f(v_{i,j})=f(v_{i-1,j-1\ (mod\ 7)})$$\\
 Since for any two successive cycles $C_{7}^{i}$ and $C_{7}^{i+1}$ of ${P_{m}\square C_{7}}$, there is not a path of length $3$ between $(v_{i,j})$ and $(v_{i+1,j+1\ (\emph{mod}\ 7)})$, we may assign same color to these two vertices.
Therefore, it is observed that $$\chi_{ei}({P_{m}\square C_{7}})=4.$$

\textbf{Case 4}. Let $m\geq2, \  n \ \textmd{is} \ \textmd{evevn}$. The Cartesian product ${P_{m}\square C_{n}}$ contains copies of $P_{m}$ and $C_{n}$ as subgraphs, so $\chi_{ei}({P_{m}\square C_{n}})\geq 2$. \\
We construct an $e$-injective coloring with $2$ colors of ${P_{m}\square C_{n}}$ using of pattern given in the matrix in Table \ref{T5}. We define a function $f:V\rightarrow\{1,2\}$ as follows.\
$$f(v_{i,j})=\begin{cases}
 1 & 2|i+j  \\
 2 & 2\nmid i+j \\
 \end{cases}.$$\\
Obviously, there is no path of length $3$ between every pair of vertices with same color $1$ or with same color $2$. Therefore $\chi_{ei}({P_{m}\square C_{n}})=2$.\\
\begin{table}[h]
  \centering
  $\begin{pmatrix}
    1 & 2 & 1 & \dots & 2\\
    2 & 1 & 2 & \dots & 1\\
    1 & 2 & 1 & \dots & 2\\
    \vdots & \vdots & \vdots &\ldots & \vdots\\
  \end{pmatrix}$
  \caption{Matrix $e$-injective coloring of ${P_{m}\square C_{n}}$ }\label{T5}
\end{table}

\textbf{Case 5}.  Let $m\geq 2, \ n$ is odd and ($n\geq 9$). By theorem \ref{prop-dn}, we know that $\chi_{ei}({P_{m}\square C_{n}})\geq 3$.
We construct an optimal $e$-injective coloring with $3$ colors for ${P_{m}\square C_{n}}$. For this, we define a function $f:V\rightarrow \{1,2,3\}$ as follows.\\
 $$f(v_{1,j})=\begin{cases}
 1 &  \ j\in \{1, 3, 5\} \bigcup \{2k| k\geq 5\}\\
 2 &  \ j\in \{2, 7, 9\} \bigcup \{2k+1| k\geq 5\}\\
 3 &  \ j\in \{4, 6, 8\}\\
 \end{cases}$$
  and
$$f(v_{i,j})=f(v_{i-1,j-1\ (mod\ n)})$$\\
Obviously, there is no path of length $3$ between two vertices with the same color. Thus, we can see that $f$ is an $e$-injective coloring.
 Therefore $\chi_{ei}({P_{m}\square C_{n}})=3$.
\end{proof}

\subsection{Tori graphs}
In this subsection, our main purpose is to give the $e$-injective chromatic number for tori graph. For this, we begin with the special case.

\begin{theorem}\label{theo-C*C*}
For cycles $C_{m}$ and $C_{n}$, $(m=3$ and $n\in\{3, 5, 7\})$, we have

$$\chi_{ei}(C_{m}\square C_{n})=\begin{cases}
  9 & \textmd{if} \ m=n=3 \\
  15 & \textmd{if} \ m=3, \  n=5 \\
  12 & \textmd{if} \ m=3, \  n=7 \\
 \end{cases}$$
\end{theorem}
\begin{proof}
According to the Figure \ref{Fig-car-dn} and Table \ref{T-11}, the proof is obvious.
\begin{figure}[h]
	\centering
		\begin{tikzpicture}[scale=.4, transform shape]
		\node [scale=1.3] at (0,1.1) {$u_{1}$};
        \node [scale=1.3] at (1.2,2.5) {$u_{2}$};
		\node [scale=1.3] at (0,3.9) {$u_{3}$};

        \node [scale=1.3] at (0.8,-0.7) {$v_{1}$};
        \node [scale=1.3] at (2,-2.3) {$v_{2}$};
		\node [scale=1.3] at (3.5,-1.9) {$v_{3}$};

        \node [scale=1.3] at (-0.8,-0.7) {$w_{1}$};
		\node [scale=1.3] at (-2.9,-0.4) {$w_{2}$};
		\node [scale=1.3] at (-3.5,-1.9) {$w_{3}$};

        \draw  (0,0) circle (1.5cm);
        \draw  (0,0) circle (2.5cm);
        \draw  (0,0) circle (3.5cm);

        \draw[line width=1pt](0,1.5) - - (0.7,2.4);
        \draw[line width=1pt](0.7,2.4) - - (0,3.5);
        \draw[line width=1pt](0,3.5) - - (0,1.5);
         \draw[line width=1pt](1.3,-0.8) - - (1.7,-1.9);
        \draw[line width=1pt](1.7,-1.9) - - (3,-1.8);
        \draw[line width=1pt](3,-1.8) - - (1.3,-0.8);
        \draw[line width=1pt](-1.3,-0.8) - - (-2.4,-0.5);
        \draw[line width=1pt](-2.4,-0.5) - - (-3,-1.8);
        \draw[line width=1pt](-3,-1.8) - - (-1.3,-0.8);
        \node [draw, shape=circle,fill=yellow] (v1) at  (0,1.5) {};
		\node [draw, shape=circle,fill=red] (u1) at  (0.7,2.4) {};
		\node [draw, shape=circle,fill=blue] (u1) at  (0,3.5) {};

	    \node [draw, shape=circle,fill=brown] (v2) at  (1.3,-0.8) {};
	    \node [draw, shape=circle,fill=magenta] (u2) at  (1.7,-1.9) {};
	    \node [draw, shape=circle,fill=teal] (v2) at  (3,-1.8) {};

        \node [draw, shape=circle,fill=green] (v3) at  (-1.3,-0.8) {};
		\node [draw, shape=circle,fill=pink] (u3) at  (-2.4,-0.5) {};
	    \node [draw, shape=circle,fill=gray] (v2) at  (-3,-1.8) {};

	\end{tikzpicture}
\qquad
	\centering
	\begin{tikzpicture}[scale=.4, transform shape]

		\node [scale=1.3] at (0,1.1) {$u_{1}$};
		\node [scale=1.3] at (1.3,2.5) {$u_{2}$};
		\node [scale=1.3] at (0,3.9) {$u_{3}$};

		\node [scale=1.3] at (1,0.5) {$v_{1}$};
        \node [scale=1.3] at (3,0) {$v_{2}$};
        \node [scale=1.3] at (3.8,1.3) {$v_{3}$};

        \node [scale=1.3] at (0.6,-0.9) {$w_{1}$};
		\node [scale=1.3] at (1.2,-2.6) {$w_{2}$};
		\node [scale=1.3] at (3.4,-2.1) {$w_{3}$};

        \node [scale=1.3] at (-0.6,-0.9) {$s_{1}$};
		\node [scale=1.3] at (-2.9,-0.7) {$s_{2}$};
		\node [scale=1.3] at (-3.3,-2) {$s_{3}$};

        \node [scale=1.3] at (-0.8,0.5) {$m_{1}$};
		\node [scale=1.3] at (-2,2.3) {$m_{2}$};
		\node [scale=1.3] at (-3.9,1.3) {$m_{3}$};

        \draw  (0,0) circle (1.5cm);
        \draw  (0,0) circle (2.5cm);
        \draw  (0,0) circle (3.5cm);

        \draw[line width=1pt](0,1.5) - - (0.8,2.4);
        \draw[line width=1pt](0.8,2.4) - - (0,3.5);
        \draw[line width=1pt](0,3.5) - - (0,1.5);
        \draw[line width=1pt](1.4,0.5) - - (3.3,1.3);
        \draw[line width=1pt](3.3,1.3) - - (2.5,0);
        \draw[line width=1pt](2.5,0) - - (1.4,0.5);
        \draw[line width=1pt](1.1,-1) - - (2.8,-2);
        \draw[line width=1pt](2.8,-2) - - (1.3,-2.2);
        \draw[line width=1pt](1.3,-2.2) - - (1.1,-1);
       \draw[line width=1pt](-1.1,-1) - - (-2.8,-2);
        \draw[line width=1pt](-2.8,-2) - - (-2.4,-0.7);
        \draw[line width=1pt](-2.4,-0.7) - - (-1.1,-1);
       \draw[line width=1pt](-1.4,0.5) - - (-3.3,1.3);
        \draw[line width=1pt](-3.3,1.3) - - (-1.8,1.8);
        \draw[line width=1pt](-1.8,1.8) - - (-1.4,0.5);
        \node [draw, shape=circle,fill=yellow] (u1) at  (0,1.5) {};
        \node [draw, shape=circle,fill=red] (u2) at  (0.8,2.4) {};
        \node [draw, shape=circle,fill=blue] (u3) at  (0,3.5) {};

		\node [draw, shape=circle,fill=brown] (v1) at  (1.4,0.5) {};
		\node [draw, shape=circle,fill=magenta] (v2) at  (2.5,0) {};
		\node [draw, shape=circle,fill=teal] (v3) at  (3.3,1.3) {};

	    \node [draw, shape=circle,fill=green] (w1) at  (1.1,-1) {};
	    \node [draw, shape=circle,fill=pink] (w2) at  (1.3,-2.2) {};
	    \node [draw, shape=circle,fill=gray] (w3) at  (2.8,-2) {};

        \node [draw, shape=circle,fill=orange] (s1) at  (-1.1,-1) {};
		\node [draw, shape=circle,fill=cyan] (s2) at  (-2.4,-0.7) {};
		\node [draw, shape=circle,fill=violet] (s3) at  (-2.8,-2) {};

        \node [draw, shape=circle,fill=black] (m1) at  (-1.4,0.5) {};
        \node [draw, shape=circle,fill=olive] (m2) at  (-3.3,1.3) {};
		\node [draw, shape=circle,fill=purple] (m3) at  (-1.8,1.8) {};

	\end{tikzpicture}
\qquad
    \centering
	\begin{tikzpicture}[scale=.4, transform shape]

		\node [scale=1.3] at (0,1.1) {$u_{1}$};
		\node [scale=1.3] at (1.3,2.5) {$u_{2}$};
		\node [scale=1.3] at (0,3.9) {$u_{3}$};

		\node [scale=1.3] at (0.6,0.9) {$v_{1}$};
		\node [scale=1.3] at (2.9,0.7) {$v_{2}$};
		\node [scale=1.3] at (3.2,2.1) {$v_{3}$};

        \node [scale=1.3] at (1,-0.1) {$w_{1}$};
        \node [scale=1.3] at (2.9,-1.1) {$w_{2}$};
		\node [scale=1.3] at (4,-0.2) {$w_{3}$};

        \node [scale=1.3] at (0.6,-0.8) {$s_{1}$};
        \node [scale=1.3] at (0.6,-2.8) {$s_{2}$};
        \node [scale=1.3] at (2.9,-2.5) {$s_{3}$};

        \node [scale=1.3] at (-0.5,-0.9) {$m_{1}$};
        \node [scale=1.3] at (-2,-2.2) {$m_{2}$};
        \node [scale=1.3] at (-0.9,-3.8) {$m_{3}$};

        \node [scale=1.3] at (-1,-0.1) {$n_{1}$};
        \node [scale=1.3] at (-3,0.4) {$n_{2}$};
        \node [scale=1.3] at (-3.9,-0.9) {$n_{3}$};

		\node [scale=1.3] at (-0.6,0.9) {$k_{1}$};
		\node [scale=1.3] at (-1.9,2.4) {$k_{2}$};
		\node [scale=1.3] at (-3.4,2) {$k_{3}$};

        \draw  (0,0) circle (1.5cm);
        \draw  (0,0) circle (2.5cm);
        \draw  (0,0) circle (3.5cm);

        \draw[line width=1pt](0,1.5) - - (0.8,2.4);
        \draw[line width=1pt](0.8,2.4) - - (0,3.5);
        \draw[line width=1pt](0,3.5) - - (0,1.5);

        \draw[line width=1pt](1.1,0.9) - - (2.7,2.1);
        \draw[line width=1pt](2.7,2.1) - - (2.4,0.8);
        \draw[line width=1pt](2.3,0.8) - - (1.1,0.9);

        \draw[line width=1pt](1.5,-0.2) - - (2.3,-1);
        \draw[line width=1pt](2.3,-1) - - (3.5,-0.2);
        \draw[line width=1pt](3.5,-0.2) - - (1.5,-0.2);

        \draw[line width=1pt](0.8,-1.2) - - (0.7,-2.4);
        \draw[line width=1pt](0.7,-2.4) - - (2.4,-2.5);
        \draw[line width=1pt](2.4,-2.5) - - (0.8,-1.2);

        \draw[line width=1pt](-0.6,-1.3) - - (-1.7,-1.8);
        \draw[line width=1pt](-1.7,-1.8) - - (-1,-3.4);
        \draw[line width=1pt](-1,-3.4) - - (-0.6,-1.3);

        \draw[line width=1pt](-1.5,-0.2) - - (-2.5,0.3);
        \draw[line width=1pt](-2.5,0.3) - - (-3.4,-0.9);
        \draw[line width=1pt](-3.4,-0.9) - - (-1.5,-0.2);

        \draw[line width=1pt](-1.1,0.9) - - (-1.5,2);
        \draw[line width=1pt](-1.5,2) - - (-3,1.7);
        \draw[line width=1pt](-3,1.7) - - (-1.1,0.9);
        \node [draw, shape=circle,fill=yellow] (u1) at  (0,1.5) {};
		\node [draw, shape=circle,fill=red] (u2) at  (0.8,2.4) {};
        \node [draw, shape=circle,fill=blue] (u3) at  (0,3.5) {};

		\node [draw, shape=circle,fill=brown] (v1) at  (1.1,0.9) {};
        \node [draw, shape=circle,fill=magenta] (v3) at  (2.4,0.8) {};
		\node [draw, shape=circle,fill=teal] (v2) at  (2.7,2.1) {};

		\node [draw, shape=circle,fill=yellow] (w1) at  (1.5,-0.2) {};
        \node [draw, shape=circle,fill=red] (w2) at  (2.3,-1) {};
        \node [draw, shape=circle,fill=blue] (w3) at  (3.5,-0.2) {};

	    \node [draw, shape=circle,fill=brown] (s1) at  (0.8,-1.2) {};
	    \node [draw, shape=circle,fill=magenta] (s2) at  (0.7,-2.4) {};
	    \node [draw, shape=circle,fill=teal] (s3) at  (2.4,-2.5) {};
	
        \node [draw, shape=circle,fill=green] (m1) at  (-0.6,-1.3) {};
		\node [draw, shape=circle,fill=pink] (m2) at  (-1.7,-1.8) {};
		\node [draw, shape=circle,fill=gray] (m3) at  (-1,-3.4) {};

        \node [draw, shape=circle,fill=orange] (n1) at  (-1.5,-0.2) {};
        \node [draw, shape=circle,fill=cyan] (n2) at  (-2.5,0.3) {};
        \node [draw, shape=circle,fill=violet] (n3) at  (-3.4,-0.9) {};

        \node [draw, shape=circle,fill=green] (k1) at  (-1.1,0.9) {};
        \node [draw, shape=circle,fill=pink] (k2) at  (-1.5,2) {};
        \node [draw, shape=circle,fill=gray] (k3) at  (-3,1.7) {};

	\end{tikzpicture}
\caption{{The graph of $C_3\square C_n $ for $n\in \{3,  5,7\}$}}\label{Fig-car-dn}
\end{figure}
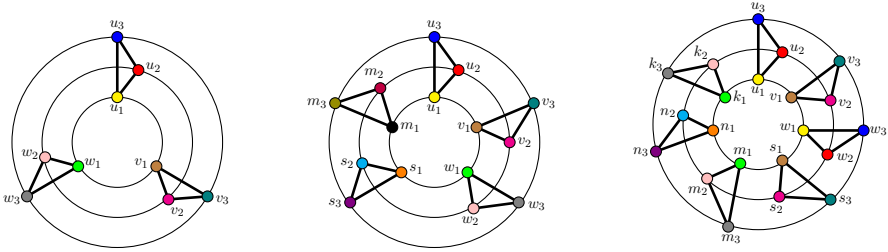

\begin{table}[h]
  \centering
  $\begin{pmatrix}
    1 & 2 & 3\\
    4 & 5 & 6 \\
    1 & 2 & 3 \\
    4 & 5 & 6 \\
    7 & 8 & 9 \\
    10 & 11 & 12\\
    7 & 8 & 9\\
  \end{pmatrix}$
  \caption{Matrix $e$-injective coloring of ${C_{3}\square C_{7}}$  }\label{T-11}
\end{table}
\end{proof}

\begin{theorem}\label{theo-C*C}
For cycles $C_{m}$ and $C_{n}$, $(m=3$ and $n$\ is\ even$)$, we have $\chi_{ei}(C_{3}\square C_{n})=6$.
\end{theorem}
\begin{proof}
Let $V(C_{3}^{i})=\{v_{i,1},v_{i,2},v_{i,3}\}$. There is a path of length $3$ between any two different vertices in $C_{3}^{i}$ and $C_{3}^{i+1}$, so $\chi_{ei}({C_{3}\square C_{n}})\geq6$.\\
 We show an $e$-injective coloring with $6$ colors of ${C_{3}\square C_{n}}$ in Table \ref{T12}. There is no path of length $3$ between two vertices $v_{i,j}$ and  $v_{i+2,j}$, so two vertices $v_{i,j}$ and $v_{i+2,j}$ can take the same color. Therefore $\chi_{ei}({C_{m}\square C_{n}})=6$.
\begin{table}[h]
  \centering
  $\begin{pmatrix}
    1 & 2 & 3\\
    4 & 5 & 6 \\
    1 & 2 & 3 \\
    4 & 5 & 6\\
    \vdots & \vdots & \vdots
  \end{pmatrix}$
  \caption{Matrix $e$-injective coloring of ${C_{3}\square C_{n}}$, $(n \ \textmd{is}\ \textmd{even})$} \label{T12}
\end{table}

\end{proof}

\begin{theorem}\label{theo-C*C}
For cycles $C_{m}$ and $C_{n}$, ($m=3$ and n\ is\ odd\, ($n\geq9$) ), we have $\chi_{ei}(C_{3}\square C_{n})=9$.
\end{theorem}
\begin{proof}
On the contrary, let we have an optimal $e$-injective coloring with less than $9$ colors. According to the Proposition \ref{prop-dn}, we need to have at least $6$ colors. We assign the colors $1,2,3$ to the vertices in $C_{3}^{1}$ and the colors $4,5,6$ to the vertices in $C_{3}^{2}$, respectively. On the other hand,  since $n$ is odd, and there is  a path of length $3$ between the vertices in $C_{3}^{n}$ with the vertices in $C_{3}^{1}$, $C_{3}^{n-1}$, so have to assign the colors $7,8,9$ to the vertices in $C_{3}^{n}$ if the vertices in $C_{3}^{n-1}$ are assigned with colors $4,5,6$ so far. This contradicts our hypothesis that $\chi_{ei}({C_{3}\square C_{n}})<9$. We conclude that $\chi_{ei}({C_{3}\square C_{n}})\geq9$.\\
We show an $e$-injective coloring with $9$ colors of ${C_{3}\square C_{n}}$ in Table \ref{T13}. This coloring is an $e$-injective coloring, because there is no path of length $3$ between any two vertices with same color. Therefore $\chi_{ei}({C_{3}\square C_{n}})=9$.\\
\begin{table}[h]
  \centering
  $\begin{pmatrix}
    1 & 2 & 3\\
    4 & 5 & 6 \\
    7 & 8 & 9 \\
    4 & 5 & 6 \\
    7 & 8 & 9 \\
    1 & 2 & 3\\
    7 & 8 & 9 \\
    1 & 2 & 3\\
    4 & 5 & 6 \\
    \vdots & \vdots & \vdots
  \end{pmatrix}$
  \caption{Matrix $e$-injective coloring of ${C_{3}\square C_{n}}$, $(n \ \textmd{is}\ \textmd{odd},\ n\geq9)$}\label{T13}
\end{table}

\end{proof}

\begin{theorem}\label{theo-C*C}
For cycles $C_{m}$ and $C_{n}$, $(m=5$ and $n\geq4)$, we have $\chi_{ei}(C_{5}\square C_{n})=5$.
\end{theorem}
\begin{proof}
Since the Cartesian product ${C_{5}\square C_{n}}$ contains ${C_{5}\square P_{2}}$ as a subgraph, so $\chi_{ei}({C_{5}\square C_{n}})\geq5$ (see Proposition \ref{prop-dn}).\\
We construct an $e$-injective coloring with $5$ colors of ${C_{5}\square C_{n}}$, (see Table \ref{T14}\ ($n \ \textmd{is}\ \textmd{even}$) and Table \ref{T15}\ ($n \ \textmd{is}\ \textmd{odd}$)).\\
 In Table \ref{T15}, there is  no path of length $3$ between any two vertices $v_{1,j}$ and $v_{2,j+1(\emph{mod}\ 5)}$ so we assign the same color to these two vertices in $C_{5}^{1}$ and $C_{5}^{2}$. It is clear that, this matrix is an $e$-injective coloring of the $({C_{5}\square C_{n}})$ when  n is odd.\\
  In Table \ref{T14},  there is  no path of length $3$ between any two vertices $v_{i,j}$ and $v_{i+2,j}$ in $C_{5}\square C_{n}$, (n is even), so two vertices $v_{i,j}, v_{i+2,j}$ can take the same color.\\
  Therefore $\chi_{ei}({C_{5}\square C_{n}})=5$.\\
\begin{table}[h]
  \centering
  $\begin{pmatrix}
    1 & 2 & 3 & 4 & 5\\
    5 & 1 & 2 & 3 & 4\\
    1 & 2 & 3 & 4 & 5\\
    5 & 1 & 2 & 3 & 4\\
   \vdots & \vdots & \vdots & \vdots & \vdots
  \end{pmatrix}$
  \caption{Matrix $e$-injective coloring of ${C_{5}\square C_{n}}$ ($n \ \textmd{is}\ \textmd{even}$)}\label{T14}
\end{table}
\begin{table}[h]
  \centering
  $\begin{pmatrix}
    1 & 2 & 3 & 4 & 5\\
    5 & 1 & 2 & 3 & 4\\
    4 & 5 & 1 & 2 & 3\\
    3 & 4 & 5 & 1 & 2\\
    2 & 3 & 4 & 5 & 1\\
   \vdots & \vdots & \vdots & \vdots & \vdots
  \end{pmatrix}$
  \caption{Matrix $e$-injective coloring of ${C_{5}\square C_{n}}$ ($n \ \textmd{is}\ \textmd{odd}$)}\label{T15}
\end{table}

\end{proof}

\begin{theorem}\label{theo-C*C}
For cycles $C_{m}$ and $C_{n}$, $(n=7, m\geq4\ (m\neq5))$, we have $\chi_{ei}(C_{m}\square C_{7})=4$
\end{theorem}

\begin{proof}
There are two cases to be considered.

\textbf{Case 1}. Let  m be odd $(m \geq7)$ and $n=7$. The Cartesian product ${C_{m}\square C_{7}}$ contains  ${P_{2}\square C_{7}}$ as a subgraph, so $\chi_{ei}({C_{m}\square C_{7}})\geq 4$.
We construct an $e$-injective coloring with $4$ colors of ${C_{m}\square C_{7}}$ in Table \ref{T16}. There is not a path of length $3$ between two vertices $v_{i,j}$ and $v_{i+1,j (\emph{mod}\ 7)}$, so $v_{i,j}$ and $v_{i+1,j (\emph{mod}\ 7)}$ can take the same color. Obviously, this matrix provides an $e$-injective coloring of the $C_{m}\square C_{7}$. Therefore $\chi_{ei}({C_{m}\square C_{7}})= 4$.
\begin{table}[h]
  \centering
  $\begin{pmatrix}
    1 & 2 & 1 & 2 & 3 & 4 & 3 \\
    3 & 1 & 2 & 1 & 2 & 3 & 4 \\
    4 & 3 & 1 & 2 & 1 & 2 & 3 \\
    3 & 4 & 3 & 1 & 2 & 1 & 2 \\
    2 & 3 & 4 & 3 & 1 & 2 & 1 \\
    1 & 2 & 3 & 4 & 3 & 1 & 2 \\
    2 & 1 & 2 & 3 & 4 & 3 & 1 \\
    \vdots&\vdots&\vdots&\vdots&\vdots&\vdots&\vdots
  \end{pmatrix}$
  \caption{Matrix $e$-injective coloring of ${C_{m}\square C_{n}}$ $(m\ is\ odd\ (m\geq7),\ n=7$)}\label{T16}
\end{table}

\textbf{Case 2}. Let m be even and $n=7$. The Cartesian product ${C_{m}\square C_{n}}$ contains ${C_{7}\square P_{2}}$ as a subgraph, so $\chi_{ei}({C_{m}\square C_{n}})\geq4$.
We shall construct an $e$-injective coloring with $4$ colors of ${C_{m}\square C_{n}}$. It easy to see that, in Table \ref{T17}, there is not a path of length $3$ between every pair of vertices with the same color. Therefore $\chi_{ei}({C_{m}\square C_{7}})=4$.

\begin{table}[h]
  \centering
  $\begin{pmatrix}
    1 & 2 & 1 & 2 & 3 & 4 & 3 \\
    2 & 1 & 2 & 1 & 4 & 3 & 4 \\
    1 & 2 & 1 & 2 & 3 & 4 & 3 \\
    2 & 1 & 2 & 1 & 4 & 3 & 4 \\
    \vdots & \vdots & \vdots & \vdots & \vdots & \vdots  & \vdots
  \end{pmatrix}$
  \caption{Matrix $e$-injective coloring of ${C_{m}\square C_{7}}$ $(m \ \textmd{is}\ \textmd{even}\ \textmd{and} \  n=7)$}\label{T17}
\end{table}

\end{proof}
\newpage

Finally, for the Cartesian product of another cycles we have.
\begin{theorem}\label{theo-C*C}
$$\chi_{ei}(C_{m}\square C_{n})=\begin{cases}
 2 & \textmd{if} \ 2|m , \  2|n   \\
 3 & \textmd{if} \  m\geq4 ,\  2\nmid n ,\ (m\neq\ 5,7\,\ n\geq9) \\
 \end{cases}$$
\end{theorem}

\begin{proof}
There are two cases to be considered.

\textbf{Case 1}. Let m and n are even. The Cartesian product ${C_{m}\square C_{n}}$ contains ${C_{m}\square P_{2}}$ as a subgraph, (see Proposition \ref{prop-dn}). Therefore $\chi_{ei}({C_{m}\square C_{n}})\geq 2$.\\
We shall construct an $e$-injective coloring with $2$ colors for ${C_{m}\square C_{n}}$ in Table \ref{T18}. It is clear that, there is no path of length $3$  between every pair of vertices with the same color. Therefore  $\chi_{ei}({C_{m}\square C_{n}})=2$.\\
\begin{table}[h]
  \centering
  $\begin{pmatrix}
    1 & 2 & 1 & 2 & 1 & 2 &\ldots \\
    2 & 1 & 2 & 1 & 2 & 1 &\ldots\\
    1 & 2 & 1 & 2 & 1 & 2 &\ldots\\
    2 & 1 & 2 & 1 & 2 & 1 &\ldots\\
    1 & 2 & 1 & 2 & 1 & 2 &\ldots\\
    2 & 1 & 2 & 1 & 2 & 1 &\ldots \\
       \vdots & \vdots & \vdots & \vdots & \vdots & \vdots
  \end{pmatrix}$
  \caption{Matrix $e$-injective coloring of ${C_{m}\square C_{n}}$($m\ \textmd{and}\ n \ \textmd{are}\ \textmd{even}$)}\label{T18}
\end{table}

\textbf{Case 2}. For case 2, we bring up two positions.

2.1. Let $m$ and $n$ be odd, ($m,n\geq 9$). The Cartesian product ${C_{m}\square C_{n}}$ contains  ${C_{m}\square P_{2}}$ as a subgraph, so $\chi_{ei}({C_{m}\square C_{n}})\geq3$.\\
We construct an $e$-injective coloring with $3$ colors of ${C_{m}\square C_{n}}$ using of pattern given in the matrix form in Table \ref{T19}. It is not hard to see that  there is no path of length 3 between every pair of vertices with the same colors. Therefore $\chi_{ei}({C_{m}\square C_{n}})= 3$.\\
\begin{table}[h]
  \centering
  $\begin{pmatrix}
    1 & 2 & 3 & 2 & 3 & 1 & 3 & 1 & 2 & 1 & 2 &\ldots & 1 & 2 \\
    2 & 1 & 2 & 3 & 2 & 3 & 1 & 3 & 1 & 3 & 1 &\ldots & 3 & 1 \\
    1 & 2 & 1 & 2 & 3 & 2 & 3 & 1 & 3 & 1 & 3 &\ldots & 1 & 3 \\
    3 & 1 & 2 & 1 & 2 & 3 & 2 & 3 & 1 & 3 & 1 &\ldots & 3 & 1 \\
    1 & 3 & 1 & 2 & 1 & 2 & 3 & 2 & 3 & 2 & 3 &\ldots & 2 & 3 \\
    3 & 1 & 3 & 1 & 2 & 1 & 2 & 3 & 2 & 3 & 2 &\ldots & 3 & 2 \\
    2 & 3 & 1 & 3 & 1 & 2 & 1 & 2 & 3 & 2 & 3 &\ldots & 2 & 3 \\
    3 & 2 & 3 & 1 & 3 & 1 & 2 & 1 & 2 & 1 & 2 &\ldots & 1 & 2 \\
    2 & 3 & 2 & 3 & 1 & 3 & 1 & 2 & 1 & 2 & 1 &\ldots & 2 & 1 \\
    1 & 2 & 3 & 2 & 3 & 1 & 3 & 1 & 2 & 1 & 2 &\ldots & 1 & 2 \\
    2 & 1 & 2 & 3 & 2 & 3 & 1 & 3 & 1 & 2 & 1 &\ldots & 2 & 1 \\
    1 & 2 & 3 & 2 & 3 & 1 & 3 & 1 & 2 & 1 & 2 &\ldots & 1 & 2 \\
    2 & 1 & 2 & 3 & 2 & 3 & 1 & 3 & 1 & 2 & 1 &\ldots & 2 & 1 \\
     \vdots & \vdots & \vdots & \vdots & \vdots & \vdots  & \vdots & \vdots & \vdots  & \vdots & \vdots & \vdots & \vdots & \vdots \\
    1 & 2 & 3 & 2 & 3 & 1 & 3 & 1 & 2 & 1 & 2 &\ldots & 1 & 2 \\
    2 & 1 & 2 & 3 & 2 & 3 & 1 & 3 & 1 & 2 & 1 &\ldots & 2 & 1 \\
  \end{pmatrix}$
  \caption{Matrix $e$-injective coloring of ${C_{m}\square C_{n}}$($m\ \textmd{and}\  n \ \textmd{are}\ \textmd{odd},\ (m,n\geq9)$)}\label{T19}
\end{table}\\

2.2. Let $m$ be even and $n$ be odd, $(n\geq9)$. The Cartesian product ${C_{m}\square C_{n}}$ contains ${C_{n}\square P_{2}}$ as a subgraph, so $\chi_{ei}({C_{m}\square C_{n}})\geq 3$.\\
We construct an $e$-injective coloring with $3$ colors of ${C_{m}\square C_{n}}$ in Table \ref{T20} and use  Proposition \ref{prop-dn} for coloring.
It is not hard to see that  there is no path of length $3$
between $v_{i,j}$ and $v_{i+2,j}$ in $C_{m}\square C_{n}$, so we assign the same color to them. Therefore $\chi_{ei}({C_{m}\square C_{n}})=3$.
\begin{table}[h]
  \centering
  $\begin{pmatrix}
    1 & 2 & 1 & 3 & 1 & 3 & 2 & 3 & 2 & 1 & 2 &\ldots & 1 & 2 & 1 & 2 \\
    2 & 1 & 2 & 1 & 3 & 1 & 3 & 2 & 3 & 2 & 1 & 2 &\ldots & 1 & 2 & 1\\
    1 & 2 & 1 & 3 & 1 & 3 & 2 & 3 & 2 & 1 & 2 &\ldots & 1 & 2 & 1 & 2 \\
    2 & 1 & 2 & 1 & 3 & 1 & 3 & 2 & 3 & 2 & 1 & 2 &\ldots & 1 & 2 & 1\\
    \vdots & \vdots & \vdots & \vdots & \vdots & \vdots  & \vdots  & \vdots & \vdots & \vdots & \vdots & \vdots & \vdots & \vdots  & \vdots  & \vdots
  \end{pmatrix}$
  \caption{Matrix $e$-injective coloring of ${C_{m}\square C_{n}}$ $(m \ \textmd{is}\ \textmd{even}\ \textmd{and} \  n \ \textmd{is}\ \textmd{odd},(n\geq9))$}\label{T20}
\end{table}

\end{proof}

\newpage
\textbf{Discussion and conclusions}
\vspace{2mm}

From Observations \ref{obs-e-uso}, \ref{obs-e-inj} and \ref{obs-e-dis},

1. Characterize graphs $G$ with $\chi_{ei}(G)=\chi(G)$;
 Characterize graphs $G$ with $\chi_{ei}(G)=\chi_i(G)$;
Characterize graphs $G$ with $\chi_{ei}(G)=\chi_2(G)$.

2. After  discussion  on the solution of 1, one can revisit the  Conjectures 1 and 2. 

From  Propositions \ref{Max deg} and \ref{prop pack}, we can have the following.

3.  Characterize graphs $G$ with $\chi_{ei}(G)=\rho(G)$.

4. Characterize graphs $G$ with $\chi_{ei}(G)=\Delta(G)(\Delta(G)-1)^2+1$.


\end{document}